\documentclass[12pt,reqno,final,a4paper]{article}

\usepackage{hyperref}
\usepackage{amsmath}
\usepackage{amsthm}
\usepackage{amssymb}
\usepackage{dsfont}
\usepackage{stmaryrd}
\usepackage{commath}
\usepackage{mathrsfs}
\usepackage{enumitem}
\usepackage{graphicx}
\font\msbm=msbm10

\numberwithin{equation}{section}

\theoremstyle{plain}
\newtheorem{theorem}{Theorem}[section]
\newtheorem{lemma}[theorem]{Lemma}
\newtheorem{corollary}[theorem]{Corollary}
\newtheorem{proposition}[theorem]{Proposition}

\theoremstyle{definition}

\newtheorem{remark}[theorem]{Remark}
\def\mathbb#1{\hbox{\msbm{#1}}}

\newcommand{\field}[1]{\ensuremath{\mathds{#1}}}

\newcommand{\N}{\field N}

\newcommand{\R}{\field R}

\newcommand{\supp}{{\rm  supp \, }}
\newcommand{\sign}{{ \ensuremath{\mbox{\rm sign} \,} }}

\renewcommand{\>}{\rangle}

\newcommand{\be}{\begin{equation}}
\newcommand{\ee}{\end{equation}}
\newcommand{\beq}{\begin{eqnarray}}
\newcommand{\beqq}{\begin{eqnarray*}}
\newcommand{\eeq}{\end{eqnarray}}
\newcommand{\eeqq}{\end{eqnarray*}}

\newcommand{\1}{\mathds{1}}

\DeclareMathOperator{\err}{err}

\newcommand{\Lip}{\ensuremath{\mathrm{Lip}}}
\newcommand{\Rs}[1]{\ensuremath{\mathcal R_d^{#1}}}
\newcommand{\Sparse}{\ensuremath{\mathfrak S^{d-1}}}
\newcommand{\Sphere}{\ensuremath{\mathds S^{d-1}}}

\newcommand{\AlgAdap}{\ensuremath{\mathcal S^{\text{ada}}}}
\newcommand{\Alg}{\ensuremath{\mathcal S}}

\usepackage{ifdraft}
\usepackage[colorinlistoftodos, linecolor=red,bordercolor=red, backgroundcolor=red,textsize=small,textwidth=3cm]{todonotes}
\usepackage{xcolor}
\usepackage{setspace}
\usepackage[normalem]{ulem}
\usepackage{comment}
\usepackage{framed}

\ifdraft{\usepackage{showkeys}}{}

\newcounter{todocounter}
\newenvironment{todobox}[3][]
{\addtocounter{todocounter}{1}\todo[caption={\protect\hypertarget{todo\thetodocounter}{}#2 #3},noline, #1]{#2 \hyperlink{todo\thetodocounter}{$\uparrow$}}\color{red}\textbf{#3} }
{\color{black}}

\specialcomment{help}
{
\section*{Help}
}
{}
\ifdraft{}{\excludecomment{help}}

\specialcomment{changes}
{
\section*{Changes}
}
{\listoftodos[TODO's etc.]}
\ifdraft{}{\excludecomment{changes}}

\specialcomment{detail}
{\textcolor{blue!50!black}{/** }\begingroup\color{blue!50!black}}
{\endgroup\textcolor{blue!50!black}{**/} }
\ifdraft{}{\excludecomment{detail}}

\newcommand{\old}[1]{\textcolor{red}{\sout{\ifdraft{#1}{}}}}
\newcommand{\new}[1]{\ifdraft{\textcolor{green!50!black}{#1}}{#1}}

\hypersetup{
   pdfauthor={Sebastian Mayer, Tino Ullrich, and Jan Vyb\'iral},
   pdfkeywords={ridge functions, entropy numbers, sampling, information-based complexity},
   pdftitle={Approximation properties of ridge functions}
}

\begin{document}

\title{Entropy and sampling numbers of classes of ridge functions}

\author{Sebastian Mayer$^a$, Tino Ullrich$^{a,}$\footnote{Corresponding author. Email: tino.ullrich@hcm.uni-bonn.de, Tel: +49 228 73 62224}, and Jan Vyb\'iral$^b$}

\maketitle

\vspace{-1cm}
\begin{center}
$^a$ Hausdorff-Center for Mathematics, Endenicher Allee 62, 53115 Bonn, Germany\\
$^b$ Department of Mathematics, Technical University Berlin, Stra{\ss}e des 17. Juni 136, 10623 Berlin, Germany\end{center}

\begin{abstract}
We study properties of ridge functions $f(x)=g(a\cdot x)$ in high dimensions $d$ from the viewpoint of approximation theory. The considered function classes consist of ridge functions such that the profile $g$ is a member of a univariate Lipschitz class with smoothness $\alpha > 0$ (including infinite smoothness), and the ridge direction $a$ has $p$-norm $\|a\|_p \leq 1$. First, we investigate entropy numbers in order to quantify the compactness of these ridge function classes in $L_{\infty}$. We show that they are essentially as compact as the class of univariate Lipschitz functions. Second, we examine sampling numbers and face two extreme cases. In case $p=2$, sampling ridge functions on the Euclidean unit ball faces the curse of dimensionality. It is thus as difficult as sampling general multivariate Lipschitz functions, a result in sharp contrast to the result on entropy numbers. When we additionally assume that all feasible profiles have a first derivative uniformly bounded away from zero in the origin, 
then the complexity of sampling ridge functions reduces drastically to the complexity of sampling univariate Lipschitz functions.
In between, the sampling problem's degree of difficulty varies, depending 
on the values of $\alpha$ and $p$. Surprisingly, we see almost the entire hierarchy of 
tractability levels as introduced in the recent monographs by Novak and Wo\'zniakowski.
\end{abstract}

\begin{help}

\begin{itemize}
 \item \new{Piece of text which is new in this version.}
 \item \old{Piece of text that should be removed for some reason.}
 \item \textcolor{blue!50!black}{/** more detailed descriptions; internal comments; $\dots$ **/}
 \item Remove the \texttt{draft} option to hide working draft annotations.
\end{itemize}

\end{help}

\begin{changes}
Changes to draft v11.
\begin{itemize}
\item Introduced the class of adaptive algorithms. Lower bounds in Section \ref{sec:sampling} now formulated in terms of adaptive algorithms.
\item I checked the rules how to use commas in English language. Hopefully, everything is correct know. To get a better feeling, I additionally compared with text books written by native English speakers.
\item At several points, I tried to make sentences more concise.
\end{itemize}
\end{changes}

\section{Introduction}
\old{Complex systems are nowadays ubiquitous in real-world applications. Quite often,}Functions depending on a large number (or even infinitely many) variables naturally appear in many real-world applications\old{when modelling these systems}. Since analytical representations are rarely available, there is a need to compute approximations to such functions or at least functionals thereof. Examples include parametric and stochastic PDEs \cite{CDMR, Schwab11}, data analysis and learning theory \cite{BvG,CZ,HTF}, quantum chemistry \cite{FHKS}, and mathematical finance \cite{PT}.

It is a very well-known fact that approximation of smooth multivariate functions in many cases suffers from the so-called \emph{curse of dimensionality}.
Especially, for fixed smoothness, the order of approximation decays rapidly with increasing dimension \cite{DL,lomavo96}.
Actually, a recent result \cite{NW_2009_2} from the area of \emph{information-based complexity} states that on the unit cube, even uniform approximation of infinitely differentiable functions is intractable in high dimensions. These results naturally lead to the search for other assumptions than smoothness which would allow for tractable approximation, but would still be broad enough to include real-world applications. There are many different conditions of this kind.  Usually, they require additional structure; for example, that the functions under consideration are tensor products or belong to some sort of weighted function space.
We refer to \cite{NW1, TWW} and \cite{NW2} for a detailed discussion of (in)tractability of high-dimensional problems.

In this work, we are interested in functions which take the form of a \emph{ridge}. This means that we look at functions where each $f$ is constant along lines perpendicular to some specific direction, say $a$. In other words, the function is of the form $f(x)=g(a\cdot  x)$, where $g$ is a univariate function called the profile. Ridge functions provide a simple, coordinate-independent model, which describes inherently one-dimensional structures hidden in a high-dimensional ambient space.

That the unknown functions take the form of a ridge is a frequent assumption in statistics, for instance, in the context of \emph{single index models}. For several of such statistical problems, minimax bounds have been studied on the basis of algorithms which exploit the ridge structure \cite{G,HJS,RWY}. Another point of view on ridge functions, which has attracted attention for more than 30 years,  is to approximate \emph{by} ridge functions. An early work in this direction is \cite{LS75}, which took motivations from computerized tomography, and in which the term ``ridge function'' was actually coined.  Another seminal paper is \cite{FrSt81}, which introduced \emph{projection pursuit regression} for data analysis. More recent works include the mathematical analysis of neural networks \cite{C2,Pinkus:NeuralNetworks}, and wavelet-type analysis  \cite{CD}. For a survey on further approximation-theoretical results, we refer the reader to \cite{Pinkus:ApproxRidge}.

For classical setups in statistics and data analysis, it is typical that we have no influence on the choice of sampling points. In contrast, problems of \emph{active learning} allow to \new{\emph{freely} choose} a limited number of samples from which to recover the function. Such a situation occurs, for instance, if sampling the unknown function at a point is realized by a (costly) PDE solver. In this context, ridge functions have appeared only recently as function models. The papers \cite{CT,CDDKP,FSV12} provide several algorithms and upper bounds for the approximation error.

We continue in the direction of active learning, addressing two questions concerning the approximation of ridge functions. First, we ask how ``complex''  the classes of ridge functions are compared to uni- and multivariate Lipschitz functions. We measure complexity in terms of \emph{entropy numbers}, a classical concept in  approximation theory. Second, we ask how hard it is to approximate ridge functions having only function values as information. Here, especially lower bounds are of interest to us. We formulate our results in terms of \emph{sampling numbers}. It should be pointed out, however, that we use a broader notion of sampling numbers than classical approximation theory does. As in the classical sense, we also consider a \emph{worst-case} setting with error measured in $L_{\infty}$. But sampling points may be chosen \emph{adaptively}.

Both for entropy and sampling numbers, we consider classes of ridge functions defined on the $d$-dimensional Euclidean unit ball. These classes are characterized by three parameters: the profiles' order of Lipschitz smoothness $\alpha > 0$ (including infinite smoothness $\alpha=\infty$); a norm parameter $0 < p \leq 2$ indicating the $\ell_p^d$-ball in which ridge directions must be contained; and a parameter $0 \leq \kappa \leq 1$ to impose the restriction $|g'(0)| \geq \kappa$ on the first derivative of all feasible profiles $g$ (of course, this last parameter makes only sense in case of $\alpha > 1$).

Regarding entropy numbers, the considered ridge function classes show a very uniform behaviour. For all possible parameter values, it turns out that they are essentially 
as compact as the class of univariate Lipschitz functions of the same order. For the sampling problem on the contrary, we find a much more diverse picture. 
On a first glance, the simple structure of ridge functions misleads one into thinking that approximating them should not be too much harder than approximating 
a univariate function. But this is far from true in general. Actually, in our specific setting, the sampling problem's degree of difficulty crucially depends 
on the constraint $|g'(0)| \geq \kappa$. If $\kappa > 0$, then it becomes possible to first recover the ridge direction efficiently. What remains then is only 
the one-dimensional problem of sampling the profile. In this scenario, the ridge structure indeed has a sweeping impact and the sampling problem is \emph{polynomially tractable}. But without the constraint on first 
derivatives and when all vectors in the domain may occur as ridge direction ($p=2$), sampling of ridge functions is essentially as hard as sampling of general Lipschitz functions over the same domain.
It even suffers from the \emph{curse of dimensionality}, as long as we have only finite smoothness of profiles. For other configurations of the parameters $\alpha$ and $p$, the sampling problem's level 
of difficulty varies in between the extreme cases of polynomial tractability and curse of dimensionality. Surprisingly, we obtain almost the entire spectrum of degrees of tractability as introduced 
in the recent monographs by Novak and Wo\'zniakowski.

The work is organized as follows. In Section \ref{sec:prelim}, we define the setting in a precise way and introduce central concepts. Section \ref{sec:entropy} then is dedicated to the study of entropy numbers for the considered function classes. Lower and upper bounds on sampling numbers are found in Section \ref{sec:sampling}. Finally, in Section \ref{sec:tractability}, 
we interpret our findings on sampling numbers in the language of information based-complexity.

\section{Preliminaries}
\label{sec:prelim}
When $X$ denotes a (quasi-)Banach space of functions,
equipped with the (quasi-)norm $\|\cdot\|_X$, we write $B_X = \{ f \in X: \; \|f\|_X < 1\}$ for the open unit ball and 
$\bar B_X$ for its closure. In case that $X=\ell_p^d(\R) = (\R,\|\cdot\|_p)$ we additionally use the notation $B_p^d$ for the open unit ball and 
$\Sphere_p$ for the unit sphere in $\ell_p^d$.

\subsection{Ridge function classes}
The specific form of ridge functions suggests to describe a class of such functions in terms of two parameters: one to determine the smoothness of profiles, the other to restrict the norm of ridge directions.

Regarding smoothness, we require that ridge profiles are Lipschitz of some order. For the reader's convenience, let us briefly recall this notion. 
Let $\Omega\subset \R^d$ be a bounded domain and $s$ be a natural number. The function space $C^{s}(\Omega)$ consists of those functions over the domain $\Omega$ which have partial derivatives up to order $s$ in the interior $\mathring{\Omega}$ of $\Omega$, and these derivatives are moreover bounded 
and continuous in $\Omega$. Formally,
\[
 C^s(\Omega) = \big\{ f: \Omega \to \R: \quad \|f\|_{C^s} := \max_{ |\gamma| \leq s} \|D^{\gamma} f\|_{\infty} < \infty \big\}, 
\]
where, for any multi-index $\gamma = (\gamma_1,\dots,\gamma_d) \in \N_0^d$, the partial differential operator $D^{\gamma}$ is given by 
\[
 D^{\gamma}f := \frac{\partial^{|\gamma|} f}{\partial x_1^{\gamma_1} \cdots \partial x_d^{\gamma_d}}\,.
\]
Here we have written $|\gamma| = \sum_{i=1}^d \gamma_i$ for the order of $D^{\gamma}$. For the vector of first derivatives we use the usual notation $\nabla f = (\partial f/\partial x_1,\dots, \partial f/\partial x_d)$. Beside $C^s(\Omega)$ we further need the space of infinitely differentiable
functions $C^{\infty}(\Omega)$ defined by 
\be\label{f62}
 C^{\infty}(\Omega) = \big\{ f: \Omega \to \R: \quad \|f\|_{C^{\infty}} := \sup_{ \gamma \in \N_0^d} \|D^{\gamma} f\|_{\infty} < \infty \big\}\,.
\ee

For a function $f:\Omega\to\R$ and any positive number $0 < \beta \leq 1$, the \emph{H\"older constant} of order $\beta$ is given by
\begin{equation}\label{beta}
|f|_{\beta}:=\sup_{\substack{x,y\in\Omega\\ x\not=y}}\frac{|f(x)-f(y)|}{2\min\{1, \|x-y\|_1\}^{\beta}}\;.
\end{equation} 
This definition immediately implies the relation 
\be\label{f64}
  |f|_{\beta} \leq |f|_{\beta'}\mbox{ if } 0<\beta < \beta' \leq 1.
\ee
Now, for any $\alpha >0$, we can define the \emph{Lipschitz space} $\Lip_{\alpha}(\Omega)$. If we let $s=\llfloor \alpha \rrfloor$ 
be the largest integer \emph{strictly less} than $\alpha$, it contains those functions in $C^s(\Omega)$  which have partial derivatives of order $s$ which are moreover H\"older-continuous of order $\beta = \alpha - s>0$. Formally,
\[
 \Lip_{\alpha}(\Omega) = \big\{ f \in C^s(\Omega): \quad \|f\|_{\Lip_{\alpha}(\Omega)} := \max \{ \|f\|_{C^s}, \ \max_{|\gamma| = s} |D^{\gamma}f|_{\beta} \} < \infty \big\}.
\]
For $s\in \N_0$ and $1 \geq \beta_2>\beta_1 > 0$ the following embeddings hold true 
\be\label{f60}
 C^{\infty}(\Omega) \subset \Lip_{s+\beta_2}(\Omega) \subset \Lip_{s+\beta_1}(\Omega) \subset C^s(\Omega) \subset \Lip_s(\Omega)\,,
\ee
where the respective identity operators are of norm one. In other words, the respective unit balls satisfy the same relation. Note that the fourth inclusion only makes sense if $s\geq 1$.  The third embedding is a trivial consequence of the definition. The second embedding follows from
the third, \old{the fourth,} and \eqref{f64}. The fourth embedding and the second imply the first. So it remains to establish the fourth embedding. 
We have to show that for every $\gamma \in \N_0^d$ with $|\gamma| = s-1$ it holds $|D^{\gamma}f|_1 \leq \|f\|_{C^s}$\,. On the one hand, Taylor's formula in $\R^d$ gives
for some $0<\theta<1$
\be\nonumber
  \begin{split}
    |D^{\gamma}f(x)-D^{\gamma}f(y)| &= |\nabla (D^{\gamma}f)(x+\theta(y-x))\cdot (x-y)| \\
    &\leq \max_{|\beta| = s}\|D^{\beta}f\|_{\infty}\cdot \|x-y\|_1\\
    &\leq \|f\|_{C^s}\|x-y\|_1\,.
  \end{split}  
\ee
On the other hand, we have $|D^{\gamma}f(x)-D^{\gamma}f(y)| \leq 2\|f\|_{C^s}$\,. Both estimates together yield 
$|D^{\gamma}f|_1 \leq \|f\|_{C^s}$\,.

Having  introduced Lipschitz spaces, we can give a formal definition of our ridge functions classes. For the rest of the paper, we fix as function domain the closed unit ball 
$$
    \Omega  = \bar B_2^d = \{x\in\R^d~:~\|x\|_2 \leq 1\}.
$$
As before, let $\alpha > 0$ denote the order of Lipschitz smoothness. Further, let $0 < p \leq 2$. We define the class of ridge functions with Lipschitz profiles as
\begin{align}\label{def_ridge}
 \Rs{\alpha,p} = \left\lbrace f: \Omega \to \R \; : \; f(x) = g(a\cdot x), \; \|g\|_{\Lip_{\alpha}[-1,1]} \leq 1 ,\; \|a\|_p \leq 1  \right\rbrace.
\end{align}
In addition, we define the class of ridge functions with infinitely differentiable profiles by
$$
 \Rs{\infty,p} = \left\lbrace f: \Omega \to \R \; : \; f(x) = g(a\cdot x), \; \|g\|_{C^{\infty}[-1,1]} \leq 1 ,\; \|a\|_p \leq 1  \right\rbrace.
$$

\noindent Let us collect basic properties of these classes.

\begin{lemma}\label{emb}
For any $\alpha > 0$ and $0 < p \leq 2$ the class $\Rs{\alpha,p}$ is contained in $\bar B_{\Lip_{\alpha}(\Omega)}$ and $\Rs{\infty,p}$ is 
contained in $\bar B_{C^{\infty}(\Omega)}$.
\end{lemma}
\begin{proof}
Let $f\in \Rs{\alpha,p}$ and $s=\llfloor\alpha\rrfloor$. Furthermore, let $\gamma \in \N_0^d$ be such that $|\gamma| \leq s$. Then, there exists $g\in \Lip_{\alpha}([-1,1])$ with
$$
    D^{\gamma}f(x) = D^{|\gamma|}g(a\cdot  x)a^{\gamma}\,,\quad x\in \Omega\,,
$$
where we used the convention $a^\gamma = \prod_{i=1}^d a_i^{\gamma_i}$\,.
Therefore, we have  
$$
  \|D^{\gamma} f\|_{\infty} \leq \|D^{|\gamma|} g\|_{\infty}\|a\|_{\infty}^{|\gamma|} \leq \|a\|_p^{|\gamma|} \leq 1\,. 
$$
If we let $s \to \infty$ this immediately implies $\Rs{\infty,p} \subset \bar B_{C^{\infty}(\Omega)}$.
Moreover, if $|\gamma| = s$ and $\beta = \alpha-s$ we obtain by H\"older's inequality for $x,y \in \Omega$
\begin{equation}\nonumber
  \begin{split}
    |D^{\gamma}f(x)-D^{\gamma}f(y)| &= |a^{\gamma}|\cdot|D^sg(a\cdot x)-D^sg(a\cdot y)|\\ 
    &\leq \|a\|_p^s\cdot |D^sg|_{\beta}\cdot 2\min\{1,\|a\|_p\cdot \|x-y\|_1\}^{\beta}\\
    &\leq 2\min\{1,\|x-y\|_1\}^{\beta}\,.
  \end{split}
\end{equation}
Consequently, we have $\|f\|_{\Lip_{\alpha}(\Omega)} \leq 1$ and hence $\Rs{\alpha,p} \subset \bar B_{\Lip_{\alpha}(\Omega)}$.
\end{proof}

\noindent Note that in the special case $\alpha=1$, we have Lipschitz-continuous profiles. Whenever $0<\alpha_1 < \alpha_2 \leq \infty$, we have $\Rs{\alpha_2, p} \subset \Rs{\alpha_1,p}$, which is an immediate consequence of \eqref{f60}. Likewise, for $p < q$ we have the relation $\Rs{\alpha, p} \subset \Rs{\alpha, q}$.

Finally, for Lipschitz smoothness $\alpha > 1$, we want to introduce a restricted version of $\Rs{\alpha,p}$, where profiles obey the additional constraint $|g'(0)| \geq \kappa > 0$. We define
\begin{align}
 \Rs{\alpha,p,\kappa} =  \{ g(a\cdot) \in \Rs{\alpha,p}: \quad |g'(0)| \geq \kappa\}.
\end{align}
Whenever we say in the sequel that we consider ridge functions with first derivatives bounded away from zero in the origin, we mean that they are contained in the class $\Rs{\alpha,p,\kappa}$ for some $0 < \kappa \le 1$.

\paragraph{Taylor expansion.}
We introduce a straight-forward, multivariate extension of Taylor's expansion on intervals to ridge functions in
$\Rs{\alpha,p}$ and functions in $\Lip_\alpha(\Omega)$. For $x,x^0 \in \mathring{\Omega}$ we define the function $\Phi_x(\cdot)$ by 
$$
    \Phi_x(t) := f(x^0+t(x-x^0))\,,\quad t\in [0,1]\,.
$$

\begin{lemma} Let $\alpha>1$ and $\alpha = s+\beta$, $s\in \N$, $0<\beta\leq 1$. Let further $f\in \Lip_{\alpha}(\Omega)$ and $x,x^0 \in \mathring{\Omega}$. Then there is a real number $\theta \in (0,1)$ such that 
$$
  f(x) = T_{s,x^0}f(x) + R_{s,x^0}f(x)\,,
$$ 
where the Taylor polynomial $T_{s,x^0}f(x)$ is given by
\begin{equation}
  T_{s,x^0}f(x) = \sum\limits_{j=0}^s \frac{\Phi_x^{(j)}(0)}{j!} = \sum\limits_{|\gamma| \leq s} \frac{D^{\gamma}f(x^0)}{\gamma!}(x-x^0)^{\gamma}
\end{equation}
and the remainder 
\begin{align}\label{remainder1}
  R_{s,x^0}f(x) &= \frac{1}{s!}\Big(\Phi_x^{(s)}(\theta)-\Phi_x^{(s)}(0)\Big)\\ 
  \label{remainder2}&= \sum\limits_{|\gamma| = s} \frac{D^{\gamma}f(x^0+\theta(x-x^0))-D^{\gamma}f(x^0)}{\gamma!}(x-x^0)^{\gamma}\,.
\end{align}

\end{lemma}

\noindent The previous lemma has a nice consequence for the approximation of functions from $\Rs{\alpha,p}$ in case $\alpha>1$ and 
$0<p\le 2$\,. Let $p'$ denote the dual index of $p$ given by  $1/\max\{p,1\} + 1/p' = 1$.

\begin{lemma}\label{lem:taylorapprox} Let $\alpha = s+\beta > 1$ and $\Omega = \bar B^d_2$.\\
{\em (i)} For $f\in \Lip_{\alpha}(\Omega)$ and $x, x^0 \in \mathring{\Omega}$ we have 
$$
    |f(x) - T_{s,x^0}f(x)| \leq 2\|f\|_{\Lip_{\alpha}(\Omega)}\frac{\|x-x^0\|_{1}^{\alpha}}{s!}\,.
$$
{\em (ii)} Let $0<p\leq 2$. Then for $f\in \Rs{\alpha,p}$ we have the slightly better estimate
$$
    |f(x) - T_{s,x^0}f(x)| \leq \frac{2}{s!}\|x-x^0\|_{p'}^{\alpha}\,.
$$
\end{lemma}
\begin{proof} To prove (i) we use \eqref{remainder2} and the definition of $\Lip_{\alpha}(\Omega)$ and estimate as follows
\begin{align}\nonumber
  |f(x) - T_{s,x^0}f(x)| &\leq \sum\limits_{|\gamma| = s} \frac{|D^{\gamma}f(x^0+\theta(x-x^0))-D^{\gamma}f(x^0)|}{\gamma!}|(x-x^0)^{\gamma}|\\
  \nonumber&\leq 2\|f\|_{\Lip_{\alpha}(\Omega)}\min\{1,\|x-x^0\|_1\}^{\beta}\cdot \sum\limits_{|\gamma|=s}
  \frac{\prod\limits_{i=1}^d |x_i-x^0_i|^{\gamma_i}}{\gamma!}\,.
\end{align}
Using mathematical induction 
it is straight-forward to verify the multinomial identity 
$$
    (a_1+\dots+a_d)^s = \sum\limits_{|\gamma| = s}\frac{s!}{\gamma!}a_1^{\gamma_1}\cdot\dots\cdot a_d^{\gamma_d}\,.
$$
Hence, choosing $a_i = |x_i-x^0_i|$ we can continue estimating 
$$
   |f(x) - T_{s,x^0}f(x)| \leq 2\|f\|_{\Lip_{\alpha}(\Omega)}\min\{1,\|x-x^0\|_1\}^{\beta} \frac{\|x-x^0\|_{1}^{s}}{s!}
$$
and obtain the assertion in (i).\\ 
For showing the improved version (ii) for functions of type $f(x) = g(a\cdot x)$ we use formula \eqref{remainder1} of the Taylor
remainder. We easily see that for $t\in (0,1)$ it holds
$$
    \Phi^{(s)}_{x}(t) = g^{(s)}\Big(a\cdot (x^0+t(x-x^0))\Big)\cdot [a\cdot (x-x^0)]^s\,.
$$
Using H\"older continuity of $g^{(s)}$ of order $\beta$ and  H\"older's inequality we see that
\begin{align}\nonumber
  |f(x) - T_{s,x^0}f(x)| &\leq \frac{1}{s!}\Big|[a\cdot (x-x^0)]^s \cdot \Bigl\{g^{(s)}\Big(a\cdot (x^0+\theta(x-x^0))\Big)-g^{(s)}(a\cdot x^0)\Bigr\}\Big|  \\
  &\nonumber\leq \frac{1}{s!}\|a\|_p^s\cdot \|x-x^0\|^s_{p'}\cdot 2\min\{1, |\theta a\cdot (x-x^0)|^{\beta}\}  \\ 
  &\nonumber\leq \frac{2}{s!}\|x-x^0\|_{p'}^{\alpha}\,.
\end{align}
The proof is complete.
\end{proof}

\subsection{Information complexity and tractability}
\label{sec:ibc} 
In this work, we want to approximate ridge functions from $\mathcal F = \Rs{\alpha,p}$ or $\mathcal F = \Rs{\alpha,p,\kappa}$ by means of deterministic sampling algorithms, using a limited amount of function values. Any allowed algorithm $S$ consists of an \emph{information map} $N_S^{\text{ada}}:{\mathcal F}\to \R^n$, and a \emph{reconstruction map} $\varphi_S:\R^n\to L_{\infty}(\Omega)$. The former provides, for $f \in \mathcal F$, function values $f(x_1),\dots,f(x_n)$ at points $x_1,\dots,x_n \in\Omega$, which are allowed to be chosen \emph{adaptively}. Adaptivity here means that $x_i$ may depend on the preceding values $f(x_1),\dots, f(x_{i-1})$. According to \cite{NW1}, we speak of \emph{standard information}. The reconstruction map then builds an approximation to $f$ based on those function values provided by the information map.

Formally, we consider the class of deterministic, adaptive sampling algorithms $\AlgAdap = \bigcup_{n \in \N} \AlgAdap_n$, where
\begin{equation*}
\begin{split}
 \AlgAdap_n = \Big\{&S: \mathcal F \rightarrow L_{\infty}(\bar B_2^d)~:~\\
 &S = \varphi_S \circ N^{\text{ada}}_S, \varphi: \R^m \rightarrow L_{\infty}, \varphi(0)=0, \; N_S^{\text{ada}}:  \mathcal F \to \R^m, \; m \leq n \Big\}\,.
\end{split}
\end{equation*}
Let us shortly comment on the restriction $\varphi(0) = 0$. Clearly, if $N_S(f) = 0$ for some $f$ then $\|f-S(f)\| = \|f\|$. Hence,  
such a function $f$ can never be well approximated by $S(f)$ since the error can not be smaller than $\|f\|$. Without the restriction 
$\varphi(0) = 0$ either the function $f$ or $-f$ is a bad function in this respect. Indeed, assume $N_S(f) = 0$ then $S(f) = S(-f) = \varphi_S(0) = b \in \R$. Then
\begin{align*}
 \| f\| &= 1/2 \| 2 f\| = 1/2 \| f - S(f) + f + S(f) \|\\
 &\leq 1/2 \big\{ \| f - S(f)\| + \| -f - S(-f) \| \big\}\\
 &\leq \max \big\{ \| f - S(f)\|, \| -f - S(-f) \| \big\}.
\end{align*}

For the given class of adaptive algorithms, the \emph{$n$-th minimal worst-case error}
\[
 g^{\text{ada}}_{n,d}(\mathcal F, L_{\infty}) := \err_{n,d}(\mathcal F, \AlgAdap,L_{\infty}) = \inf \big\{\sup_{f \in \mathcal F} \|f-S(f)\|_{\infty} : S \in \AlgAdap_n \big\},
\]
describes the approximation error which the best possible adaptive algorithm at most makes for a given budget of sampling points and any function from $\mathcal F$. Stressing that function values are the only available information, we refer to $g^{\text{ada}}_{n,d}(\mathcal F, L_{\infty})$ as the $n$-th \emph{ (adaptive) sampling number}. To reveal the effect of adaption, it is useful to compare adaptive algorithms with the subclass $\Alg \subset \AlgAdap$ of \emph{non-adaptive}, deterministic algorithms; that is, for each algorithm $S \in \Alg$ the information map is now of the form $N_S = (\delta_{x_1},\dots, \delta_{x_n})$, with $n \in \N$ and $x_1,\dots,x_n \in \bar B_2^d$. This corresponds to \emph{non-adaptive standard information} in \cite{NW1}. The associated $n$-th worst-case error
\[
 g_{n,d}(\mathcal F,L_{\infty}) := \inf_{S  \in \mathcal S_n} \sup_{f \in \mathcal F} \norm{f - S(f)}_{\infty} = \err_{n,d}(\mathcal F, \mathcal S_n, L_{\infty}). 
\]
coincides with the standard $n$-th \emph{sampling number} as known from classical approximation theory.
As a third restriction, let us introduce the $n$-th \emph{linear} sampling number $g_{n,d}^{\text{lin}}(\mathcal F, L_{\infty})$; here, only algorithms from $\Alg$ with linear reconstruction map are allowed. Clearly,
\[
 g_{n,d}^{\text{ada}}(\mathcal F, L_{\infty}) \leq g_{n,d}(\mathcal F, L_{\infty}) \leq g_{n,d}^{\text{lin}}(\mathcal F, L_{\infty}).
\]

\begin{remark}
There are results, see \cite[Section 4.2]{NW1}, which show that neither adaptivity nor non-linearity of algorithms does help under rather general conditions. These include two conditions which are certainly not met in our setting: (1) we only have function values as information, not general linear functionals; (2) the considered ridge function classes $\Rs{\alpha,p}$ and $\Rs{\alpha,p,\kappa}$ are not convex (however, they are at least symmetric). Nevertheless, the analysis in Section \ref{sec:sampling} reveals that in our setting, both adaptivity and non-linearity cannot lead to any substantial improvement in the approximation of ridge functions. 
\end{remark}

Whenever we speak of sampling of ridge functions in the following, we refer to the problem of approximating ridge functions in $\mathcal F$ by sampling algorithms from $\AlgAdap$, the $L_{\infty}$-approximation error measured in the worst-case. Its \emph{information complexity} $n(\varepsilon,d)$ is given
for $0<\varepsilon\le 1$ and $d\in\N$ by
\[
 n(\varepsilon,d) := \min \{n \in \N: g^{\text{ada}}_{n,d}(\mathcal F,L_{\infty}) \leq \varepsilon\}.
\] 

\subsection{Entropy numbers}
The concept of entropy numbers is central to this work. They can be understood as a measure to quantify the compactness of a set w.r.t.\ some reference space. For a detailed exposure and historical remarks, we refer to the monographs \cite{CaSt90,EdTr96}.
The $k$-th entropy number $e_k(K,X)$ of a subset $K$ of a (quasi-)Banach space $X$ is defined as
\be\label{defi:entropy}
    e_k(K,X)=\inf\Big\{\varepsilon>0:K \subset \bigcup\limits_{j=1}^{2^{k-1}} (x_j+\varepsilon \bar B_X)
    \text{ for some } x_1,\dots,x_{2^{k-1}}\in X\Big\}\,.
\ee
Note that $e_k(K,X) = \inf \{ \varepsilon >0: N_{\varepsilon}(K, X) \leq 2^{k-1}\}$ holds true,  where
\begin{align}\label{defi:covering_number}
 N_{\varepsilon}(K,X) := \min \Big\{n \in \N: \quad \exists x_1,\dots, x_n \in X: \; K \subset \bigcup_{j=1}^n (x_j + \varepsilon \bar B_X)\Big\}
\end{align}
denotes the \emph{covering number} of the set $K$ in the space $X$, which is the minimal natural number $n$ such that there is an $\varepsilon$-net of $K$ in $X$ of $n$ elements. We can introduce entropy numbers for operators, as well. The $k$-th entropy number $e_k(T)$ of an operator $T:X\to Y$ between two quasi-Banach spaces $X$ and $Y$ is defined by
\be\label{f1}
e_k(T)=e_k(T(\bar B_X),Y).
\ee

The results in Section \ref{sec:entropy} and \ref{sec:sampling} rely to a great degree on entropy numbers of 
the identity operator between the two finite dimensional spaces
$X = \ell^d_p(\R)$, and $Y = \ell^d_q(\R)$. Thanks to \cite{Sch84, EdTr96, Tr97, Ku01}, their behavior is completely understood. For the reader's convenience, we restate the result.
\begin{lemma}\label{lem:schuett} Let $0<p\leq q \leq \infty$ and let $k$ and $d$ be natural numbers. Then,
$$
e_k(\bar B_p^d,\ell_q^d) \asymp \left\{
\begin{array}{rcl}1&:&1\leq k \leq \log(d),\\
\Big(\frac{\log(1+d/k)}{k}\Big)^{1/p-1/q}&:&\log(d)\leq k\leq d,\\
2^{-\frac{k}{d}}d^{1/q-1/p}&:&k\geq d\,.
\end{array}
\right.
$$
The constants behind ``$\asymp$'' do neither depend on $k$ nor on $d$. They only depend on the parameters $p$ and $q$.
\end{lemma}

If we consider entropy numbers of $\ell_p^d$-spheres instead of $\ell_p^d$-balls in $\ell_q^d$, the situation is quite similar. We are not aware of a reference where this has already been formulated thoroughly. 
\begin{lemma}\label{res:entropy_sphere}
Let $d\in \N$, $d\geq 2$, $0<p\leq q \leq \infty$, and $\bar p = \min\{1,p\}$. Then,
\begin{enumerate}[label=(\roman*)]
\item \[ 2^{-k/(d-1)} d^{1/q-1/p} \; \lesssim \; e_k(\Sphere_p, \ell_q^d) \; \lesssim \; 2^{-k/(d-\bar p)} d^{1/q-1/p}, \quad k \geq d. \]
\item
$$
e_k(\Sphere_p,\ell_q^d) \asymp \left\{
\begin{array}{rcl}1&:&1\leq k \leq \log(d),\\
\Big(\frac{\log(1+d/k)}{k}\Big)^{1/p-1/q}&:&\log(d)\leq k\leq d.
\end{array}
\right.
$$
\end{enumerate}
The constants behind ``$\asymp$'' only depend on $p$ and $q$.

\end{lemma}

\begin{proof}
For given $\varepsilon > 0$, an $\varepsilon$-covering $\{y_1,\dots,y_N\}$ of $\Sphere_p$ in $\ell_p^d$ fulfils
\begin{align}\label{eq:entropy_sphere_1}
 (1+\varepsilon) \bar B_p^d \setminus (1-\varepsilon) \bar B_p^d \subseteq &
 \bigcup_{i=1}^N (y_i + 2^{1/\bar p} \varepsilon \bar B_p^d).
\end{align}
Let $\bar q = \min\{1,q\}$. For given $\varepsilon > 0$, a maximal set $\{x_1,\dots,x_M\} \subset \Sphere_p$ of vectors with mutual distance greater $\varepsilon$ obeys
\begin{align}\label{eq:entropy_sphere_2}
 & \bigcup_{i=1}^M (x_i + 2^{-1/\bar q} \, \varepsilon \bar B_q^d) \subseteq
 (1 + \varepsilon_d^{\bar p})^{1/\bar p} \bar B_p^d \setminus (1 - \varepsilon_d^{\bar p})^{1/\bar p} \bar B_p^d,
\end{align}
where $\varepsilon_d = 2^{-1/\bar q} \, \varepsilon \, d^{1/p-1/q}$.

\noindent\emph{(i).} A standard volume argument applied to \eqref{eq:entropy_sphere_1} yields
$h(\varepsilon) \leq N \varepsilon^d 2^{d/\bar p}$, where $h(\varepsilon) = (1+\varepsilon)^d - (1-\varepsilon)^d$. First-order Taylor expansion in $\varepsilon$ allows to estimate $h(\varepsilon) \geq d\varepsilon$. Solving for $N$ yields a lower bound for covering numbers in case $p=q$. The lower bound in case $p\neq q$ follows from the trivial estimate $e_k(\Sphere_p, \ell_q^d)  \geq d^{1/q-1/p} \; e_k(\Sphere_p,\ell_p^d)$.

For the upper bound in case $p=q$ a standard volume argument applied to \eqref{eq:entropy_sphere_2} yields $M \varepsilon^d 2^{-d/\bar p} \leq h_{p}(\varepsilon^{\bar p}/2)$ with $h_{p}(x) = (1+x)^{d/\bar p} - (1-x)^{d/\bar p}$. The mean value theorem gives $h(x) \leq d/\bar p\, 2^{d/\bar p}x$ if $0<x\leq 1$. Hence, we get $h_{p}(\varepsilon^{\bar p}/2) \leq d/\bar p\, 2^{d/\bar p} \varepsilon^{\bar p}/2$. Solving for $M$ gives an upper bound for packing numbers and hence also for covering numbers. In case $p\neq q$ we again use
\eqref{eq:entropy_sphere_2} and pass to volumes. This time the quotient $\mbox{vol}(B_p^d)/\mbox{vol}(B_q^d)$ remains in the upper
bound for $M$. The given bounds now easily translate to the stated bounds on entropy numbers. In case $p\neq q$
one has to take
$$
  \Big[\frac{\mbox{vol}(B_p^d)}{\mbox{vol}(B_q^d)}\Big]^{1/(d-\bar p)} \asymp d^{1/q-1/p}
$$
into account to get the additional factor in $d$.

\noindent\emph{(ii).} The proof by K\"uhn \cite{Ku01} immediately gives the lower bound. The upper bound follows trivially from $\Sphere_p \subset \bar B_p^d$.

\end{proof}

\begin{remark} Note, that in case $p \geq 1$ we have the sharp bounds
$$
e_k(\Sphere_p,\ell_q^d) \asymp \left\{
\begin{array}{rcl}1&:&1\leq k \leq \log(d),\\
\Big(\frac{\log(1+d/k)}{k}\Big)^{1/p-1/q}&:&\log(d)\leq k\leq d,\\
2^{-\frac{k}{d-1}}d^{1/q-1/p}&:& k \geq d\,.
\end{array}
\right.
$$
In case $p<1$ there remains a gap between the upper and lower estimate for $e_k(\Sphere_p,\ell_q^d)$ if $k\geq d$.
However, this gap can be closed by using a different proof technique, see \cite{HiMa13}.
\end{remark}

\section{Entropy numbers  of ridge functions}
\label{sec:entropy}

This section is devoted to the study of entropy numbers of the classes $\Rs{\alpha,p}$ and $\Rs{\alpha,p,\kappa}$. 
Especially, we want to relate their behavior to that of entropy numbers of uni- and multivariate Lipschitz functions. This will give us an understanding how ``large'' the ridge function classes are. Let us stress that we are interested in the dependence of the entropy numbers on the underlying dimension $d$, as it is usually done in the area of information-based complexity.

To begin with, let us examine uni- and multivariate Lipschitz functions from $\Lip_{\alpha}[-1,1]$ and $\Lip_{\alpha}(\Omega)$.
Recall the notation $B_{\alpha}:=B_{\Lip_{\alpha}[-1,1]}$ and $B_{\Lip_{\alpha}(\Omega)}$ for the respective unit balls. 
The behavior of entropy numbers of univariate Lipschitz functions is well-known, see for instance \cite[Chap. 15, \S 2, Thm. 2.6]{lomavo96}.

\begin{lemma}\label{lem:en_univariate_Lip} For $\alpha>0$ there exist two constants $0<c_{\alpha}<C_{\alpha}$ such 
that 
\[
c_{\alpha} k^{-\alpha} \leq e_k(\bar B_{\alpha},L_{\infty}([-1,1])) \leq C_{\alpha}k^{-\alpha}\,,\quad k\in \N\,.
\]
\end{lemma}

\noindent This behavior does not change if we consider only functions with first derivative in the origin bounded away from zero, as we do with the profiles in the class $\Rs{\alpha,p,\kappa}$.

\begin{proposition}\label{res:en_univariate_Lip_derivative} Let $\alpha > 1$ and $0 < \kappa \le 1$. Consider the class
\[
 \Lip_{\alpha}^{\kappa}(\intcc{-1,1}) = \{ f \in \Lip_{\alpha}(\intcc{-1,1}): \|f\|_{\Lip_{\alpha}[-1,1]}\leq 1~,~ \abs{f'(0)} \geq \kappa\}.
\]
For the entropy numbers of this class we have two constants $0<c_{\alpha}<C_{\alpha}$, such that 
\[
 c_{\alpha} k^{-\alpha} \leq e_k(\Lip_{\alpha}^{\kappa}([-1,1]), L_{\infty}([-1,1])) \leq C_{\alpha} k^{-\alpha}\,,\quad k\in \N\,.
\]
\end{proposition}


\begin{proof}
The upper bound is immediate by Lemma \ref{lem:en_univariate_Lip}. The lower bound is proven in the same way as for general univariate Lipschitz functions of order $\alpha$ 
except that we have to adapt the ``bad'' functions such that they meet the constraint on the first derivative in the origin.
Put again $s=\llfloor\alpha\rrfloor$ and $\beta = \alpha -s>0 $. Consider the standard smooth bump function
\be\label{eq:bump'}
 \varphi(x) = \begin{cases} e^{-\frac{1}{1-x^2}}\quad &: \ |x|< 1,\\
 0\quad &:\ |x|\ge 1.
 \end{cases}
\ee
Let
\[
 \psi_{k,b}(x) = \frac{c_{\alpha} \cdot \varphi(5k(x-b))}{k^{\alpha}},\quad k\in \N,\ b\in \R,
\]
where $c_{\alpha} = 1/(5^\alpha\norm{\varphi}_{\Lip_\alpha})$.
The scaling factor $c_{\alpha} k^{-\alpha}$ assures $\psi_{k,b} \in \Lip_{\alpha}(\intcc{-1,1})$. Let $a=\pi/4-1/5$ and $I=[a, a+2/5]\subset(0,1)$. We put $h(x)=\sin(x)$ and 
\begin{align}\label{eq:maxcosin}
\gamma=\sup_{j\in{\N}_0}\max_{x\in I}|h^{(j)}(x)|=\max_{x\in I}\max \{\cos(x),\sin(x) \}<1.
\end{align}
For any multi-index $\theta = (\theta_1, \dots, \theta_k)\in \{0,1\}^k$ let
\[
 g_{\theta} = (1-\gamma)\sum_{j=1} ^k \theta_j \psi_{k,b_j}, \qquad b_j = a + \frac{2j-1}{5k}.
\]
Observe, that $\supp g_{\theta}\subset I.$

There are $2^k$ such multi-indices and for two different multi-indices $\hat \theta$ and $\tilde \theta$ we have 
\[
 \norm{g_{\hat \theta} - g_{\tilde \theta}}_{\infty} = (1-\gamma) \norm{\psi_{k,0}}_{\infty} = c_{\alpha} (1-\gamma)e^{-1}k^{-\alpha}.
\]
Put $f_{\theta} = h + g_{\theta}$.
Because of the scaling factors, it is assured that $f_{\theta} \in \Lip_{\alpha}^{\kappa}(\intcc{-1,1})$.
On the other hand, $f'_\theta(0)=\cos(0)=1.$
Obviously, $\norm{f_{\tilde \theta} - f_{\hat \theta}}_{\infty} = \norm{g_{\tilde \theta} - g_{\hat \theta}}_{\infty}$.
We conclude
\[
 e_k(\Lip_{\alpha}^{\kappa}(\intcc{-1,1}), L_{\infty}) \geq c'_{\alpha} k^{-\alpha}
\]
for $c'_\alpha=(1-\gamma)e^{-1}c_\alpha.$

\end{proof}

Considering multivariate Lipschitz functions, decay rates of entropy numbers change dramatically compared to those of univariate Lipschitz functions; they depend exponentially on $1/d$.
This is known if the domain is a cube $\Omega=I^d$, see \cite[Chap. 15, \S 2]{lomavo96}. We provide an extension to our situation where 
the domain is $\Omega = \bar B_2^d$.

\begin{proposition}
Let $\alpha > 0$. For natural numbers $n$ and $k$ such that $2^{k-1} < n \leq 2^k$ we have
\[
 e_n(id: \Lip_{\alpha}(\bar B_2^d) \to L_{\infty}(\bar B_2^d)) \geq c_{\alpha} e_{k+1}(id: \ell_2^d \to \ell_2^d)^{\alpha}.
\]
In particular, \old{for $n \geq 2^d$,} we have $e_n(id: \Lip_{\alpha}(\bar B_2^d) \to L_{\infty}(\bar B_2^d)) \gtrsim n^{-\alpha/d}$.
\end{proposition}
\begin{proof}
Consider the radial bump function $\varphi(x)$ given by
\be\label{eq:bump}
 \varphi(x) = \begin{cases} e^{-\frac{1}{1-\|x\|_2^2}}\quad &: \ \|x\|_2< 1,\\
 0\quad &:\ \|x\|_2\ge 1.
 \end{cases}
\ee
Let $s=\llfloor\alpha\rrfloor$. With $c_{\alpha} := (\norm{\varphi}_{\Lip_\alpha})^{-1}$ the rescaling
\[
 \varphi_{\varepsilon}^{\alpha}(x) := c_{\alpha} \varepsilon^{\alpha} \varphi(x/\varepsilon)
\]
is contained in the closed unit ball of $\Lip_{\alpha}(\Omega)$.

For $\varepsilon > 0$ let $\{x_1, \dots, x_n\}$ be a maximal set of $2\varepsilon$-separated points in the Euclidean ball $\bar B_2^d$, the distance measured in $\ell_2^d$. For every multi-index $\theta \in \{0,1\}^n$, we define
\[
 f_{\theta}(x) := \sum_{j=1}^n \theta_j \varphi^{\alpha}_{\varepsilon}(x-x_j).
\]
By construction of $\varphi_{\varepsilon}^{\alpha}$, it is assured that $f_{\theta} \in \Lip_{\alpha}(\Omega)$ and $\norm{f_{\theta}}_{\Lip_{\alpha}} \leq 1$. Moreover, we see immediately that $\norm{f_{\theta}}_{\infty}=c_{\alpha} e^{-1}\varepsilon^{\alpha}$, and
\[
 \norm{f_{\theta} - f_{\theta'}}_{\infty} \ge  c_{\alpha} e^{-1}\varepsilon^{\alpha} =: \varepsilon_1
\]
for $\theta \neq \theta'$. Therefore, the set $\{f_{\theta}: \theta \in \{0,1\}^n\}$ consists of $2^n$ functions with mutual distances greater than or equal to $\varepsilon_1$.

Now choose $\varepsilon$ such that
\[
 e_{k+1}(\bar B_2^d, \ell_2^d) < 2\varepsilon < e_k(\bar B_2^d, \ell_2^d).
\]
Then, for $n$ as above, we have $2^k \geq n > 2^{k-1}$, and
\[
 2^{n-1} < N_{\varepsilon_1/2}(\bar B_{\Lip_{\alpha}(\Omega)},L_{\infty}).
\]
We conclude
\[
 e_n(id: \Lip_{\alpha}(\Omega) \to L_{\infty}(\Omega)) > \varepsilon_1/2 > c'_{\alpha} e_{k+1}(id: \ell_2^d \to \ell_2^d)^{\alpha}
\]
for $c_\alpha'=c_\alpha/(4e).$
Now it follows immediately from the estimate above and Lemma \ref{lem:schuett} that
\[
 e_n(id: \Lip_{\alpha}(\Omega) \to L_{\infty}(\Omega)) \gtrsim 2^{-\alpha(k+1)/d} \gtrsim n^{-\alpha/d}.
\]
 
\end{proof}

Now consider ridge functions with Lipschitz profile as given by the class $\Rs{\alpha,p}$.

\begin{theorem}\label{res:ridge_entropy}
Let $d$ be a natural number, $\alpha > 0$, and $0 < p \leq 2$. Then, for any $k\in \N$,
\[
 \frac 12 \max \{ e_{2k}(\bar B_p^d,\ell_2^d), e_{2k}(\bar B_{\alpha},L_\infty) \}
 \leq e_{2k}(\Rs{\alpha,p},L_\infty)
 \leq
 e_{k}(\bar B_p^d,\ell_2^d)^{\min\{\alpha,1\}} + e_{k}(\bar B_{\alpha},L_\infty) \;.
\]
\end{theorem}


\begin{proof}
\emph{Lower bounds:} For $\varepsilon > 0$ let $g_1,\dots,g_n$ be a maximal set of univariate Lipschitz functions in $\bar B_{\alpha}$ with mutual distances $\norm{g_i - g_j}_{\infty} > \varepsilon$ for $i\not=j$. 
Now, let $a = (1,0,\dots,0)$ and put $f_i(x) = g_i(a\cdot x)$ for $i=1,\dots,n$.
Then, of course, we have $f_i \in \Rs{\alpha,p}$, and
$$
\norm{f_i - f_j}_{\infty} = \|g_i-g_j\|_{\infty} > \varepsilon\,.
$$
Consequently, the functions $f_1, \dots, f_n$ are $\varepsilon$-separated, as well.
This implies 
$$
e_{2k}(\Rs{\alpha,p},L_\infty) \ge\, \frac12\, e_{2k}(\bar B_{\alpha},L_{\infty})\,.
$$
\begin{detail}
For arbitrary $\delta > 0$, choose $\delta > 0$ such that
\[
 e_{2k}(\bar B_{\alpha},L_{\infty}) -\delta < 2 \varepsilon < e_{2k}(\bar B_{\alpha},L_{\infty}).
\]
Then,
\[
 N_{\varepsilon}(\Rs{\alpha,p}) \geq M_{2\varepsilon}(\Rs{\alpha,p}) \geq M_{2\varepsilon}(\bar B_{\alpha}) \geq N_{2\varepsilon}(\bar B_{\alpha}) > 2^{2k-1}.
\]
By definition of entropy numbers, this implies $\varepsilon < e_{2k}(\Rs{\alpha,p})$. Letting $\delta \to 0$, this further implies
\[
 \frac 12 e_{2k}(\bar B_{\alpha}) \leq e_{2k}(\Rs{\alpha,p}). 
\]
\end{detail}

On the other hand, for $\varepsilon > 0$,  let $a_1, \dots, a_n$ be a maximal set of vectors in $\bar B_p^d$ with pairwise distances $\norm{a_i-a_j}_2 > \varepsilon$. Furthermore,
let $g(t)=t$ and put $ \tilde f_i(x) = g(a_i\cdot  x)$ for $i=1,\dots,n$. Then $\tilde f_i \in \Rs{\alpha,p}$ and
\begin{align*}
\|\tilde f_i- \tilde f_j\|_\infty&=\sup_{x\in \bar B_2^d}|\tilde f_i(x)-\tilde f_j(x)|=\sup_{x\in \bar B_2^d}|g(a_i\cdot  x)-g(a_j\cdot  x)|\\
&=\sup_{x\in \bar B_2^d}|(a_i-a_j)\cdot x|=\|a_i-a_j\|_2 > \varepsilon.
\end{align*}
Thus, the functions $\tilde f_1, \dots, \tilde f_n$ are $\varepsilon$-separated w.r.t. the $L_{\infty}$-norm. This implies
$$
e_{2k}(\Rs{\alpha,p},L_\infty) \ge\,\frac{1}{2}\, e_{2k}(\bar B_p^d,\ell_2^d)\,.
$$

\emph{Upper bound:} Let $1/2>\varepsilon_1, \varepsilon_2>0$ be fixed and put $\varepsilon := \varepsilon_1^{\bar \alpha} + \varepsilon_2$.
Let $\mathcal N=\{g_1,\dots,g_n\}$ be a minimal $\varepsilon_1$-net of $\bar B_{\alpha}$ in the $L_\infty$-norm.
Further, let $\mathcal M=\{a_1,\dots,a_m\}$ be a minimal $\varepsilon_2$-net  of $\bar B_p^d$ in the $\ell_2^d$-norm.

Now, fix some ridge function $f:x \mapsto g(a\cdot  x)$ in $\Rs{\alpha,p}$, i.e. $\|g\|_{\Lip_\alpha}\le 1$ and $\|a\|_p\le 1$.
Then there is a function $g_i\in{\mathcal N}$ with $\|g-g_i\|_\infty\le \varepsilon_1$ and a vector $a_j\in{\mathcal M}$ with $\|a-a_j\|_2\le\varepsilon_2.$ 
Putting this together and writing $\bar\alpha=\min\{\alpha,1\}$, we obtain
\begin{align*}
\|g(a\cdot  x)-g_i(a_j\cdot  x)\|_\infty&\le \sup_{x\in \bar B_2^d}|g(a\cdot  x)-g(a_j\cdot  x)|+|g(a_j\cdot  x)-g_i(a_j\cdot  x)|\\
&\le \sup_{x\in \bar B_2^d} |g|_{\bar\alpha}\cdot|a\cdot  x-a_j\cdot  x|^{\bar\alpha}+\|g-g_i\|_\infty\\
&\le \|a-a_j\|^{\bar\alpha}_2 +\|g-g_i\|_\infty\le \varepsilon_1^{\bar\alpha} + \varepsilon_2 = \varepsilon.
\end{align*}
Hence, the set
$
\{x\to g(a\cdot  x):g\in {\mathcal N}, a\in{\mathcal M}\}
$
is an $\varepsilon$-net of $\Rs{\alpha,p}$ in $L_\infty(\Omega)$ with cardinality
\[
 \#{\mathcal N}\cdot\#{\mathcal M} = N_{\varepsilon_1}(\bar B_{\alpha},L_{\infty}) \cdot N_{\varepsilon_2}(\bar B_p^d, \ell_2^d) \;.
\]
Consequently, $N_{\varepsilon}(\Rs{\alpha,p},L_{\infty}) \leq \#{\mathcal N}\cdot\#{\mathcal M}$ and we conclude that
\[
 e_{2k}(\Rs{\alpha,p},L_{\infty}) \leq e_{k}(\bar B_p^d, \ell_2^d)^{\bar\alpha} + e_{k}(\bar B_{\alpha}, L_{\infty}) \;.
\]
\begin{detail}
For arbitrary $\delta > 0$, choose $\varepsilon_1, \varepsilon_2$ such that
\begin{align*}
 e_{k}(id: \ell_p^d \to \ell_2^d)  &< \varepsilon_1 < e_{k}(id: \ell_p^d \to \ell_2^d) + \delta,\\
 e_{k}(\bar B_{\alpha},L_{\infty}) &< \varepsilon_2 < e_{k}(\bar B_{\alpha},L_{\infty}) + \delta.
\end{align*}
Because of $N_{\varepsilon_1+ \varepsilon_2}(\Rs{\alpha,p}, L_{\infty} \leq 2^{2k-2} < 2^{2k-1}$, we obtain in the limit $\delta \to 0$
\[
 e_{2k}(\Rs{\alpha,p}) \leq \varepsilon_1 + \varepsilon_2 \leq e_{k}(id: \ell_p^d \to \ell_2^d) + e_{k}(\bar B_{\alpha}, L_{\infty}).
\]
\end{detail}
\end{proof}

\begin{remark}
In view of Proposition \ref{res:en_univariate_Lip_derivative}, it is easy to see that Theorem \ref{res:ridge_entropy} keeps valid
if we replace the class $\Rs{\alpha,p}$ by $\Rs{\alpha,p,\kappa}$.
\end{remark}

We exemplify the consequences of Theorem \ref{res:ridge_entropy} by considering the case $p=2$; for $0 < p < 2$ estimates would be similar. As the corollary below shows, entropy numbers of ridge functions asymptotically decay as fast as those of their profiles. In contrast to multivariate Lipschitz functions on $\Omega$, the dimension $d$ does not appear in the decay rate's exponent. It only affects how long we have to wait until the asymptotic decay becomes visible.

\begin{corollary}\label{thm33} Let $d$ be a natural number and $\alpha > 0$. For the entropy numbers of $\Rs{\alpha,2}$ in $L_{\infty}(\Omega)$ we have
\be\label{eq:entr}
\max(k^{-\alpha},2^{-k/d})\lesssim e_k(\Rs{\alpha,2},L_\infty)\lesssim
\begin{cases}
  1 &                 : k \le c_{\alpha} d\log d,\\
  k^{-\alpha} & : k \ge c_{\alpha} d\log d\,,
\end{cases}
\ee
for some universal constant $c_{\alpha}>0$ which does not depend on $d$.
\end{corollary}

Before we turn to the proof, let us note that \eqref{eq:entr} implies that
$$
e_k(\Rs{\alpha,2},L_\infty) \asymp 1\quad \text{if}\quad k\le d,
$$
and
$$
e_k(\Rs{\alpha,2},L_\infty) \asymp k^{-\alpha} \quad \text{if}\quad k\ge  c_{\alpha} d\ln d.
$$
Hence, entropy numbers of ridge functions are guaranteed to decay like those of their profiles for $k\ge c_{\alpha}d\log d$---and surely behave differently for $k\le d$.

\begin{proof}[Proof of Corollary \ref{thm33}]
The lower bound in \eqref{eq:entr} follows from Theorem \ref{res:ridge_entropy} combined with Lemma \ref{lem:schuett}, and Lemma \ref{lem:en_univariate_Lip}.
The upper bounds are proven in the same manner, using the simple fact that for every $\alpha>0$ there are two constants $c_\alpha,c'_\alpha>0$, such that $k\ge c_\alpha d\log d$ implies that $2^{-\min\{\alpha,1\}k/d}\le c_\alpha'k^{-\alpha}.$

\begin{detail}
To see the upper bound, let $\mathcal M$, $\mathcal N$ as in the proof of Theorem \ref{res:ridge_entropy}. It follows from \cite[Chapter 15]{lomavo96} that one may assume that $\#{\mathcal M}\le (3/\varepsilon)^d$. Furthermore, by Lemma \ref{lem:en_univariate_Lip}, we observe that there is a constant $c>0$, such that 
\[
 \#{\mathcal N}\le 2^{c\varepsilon^{-1/\alpha}}.
\]

In detail, the argument is as follows. There is a constant $c' > 0$ such that $e_k < c'k^{-\alpha}$. By definition of entropy numbers, for all $\delta > e_k$ we have $N_{\delta}(B_{\text{Lip}\del{\intcc{-1,1}}}) \leq 2^k$. Now put $\varepsilon = c'k^{-\alpha}$, $c = c'^{1/\alpha}$ and choose $\mathcal N$ an optimal $\varepsilon$-net.

Hence, for every $1/2>\varepsilon>0$, there is a $2\varepsilon$-net of $\Rs{\alpha,2}$ in $L_\infty(\Omega)$ with cardinality
at most $(3/\varepsilon)^d\cdot 2^{c\varepsilon^{-1/\alpha}}.$

Let us convert this into an upper bound for the entropy numbers of $\Rs{\alpha,2}$.
Let $\varepsilon_0 >0$ such that $(6/\varepsilon_0)^{1/\alpha} = d\log d$. Taking the logarithm gives the estimate $\log d \geq (2\alpha)^{-1}\log(6/\varepsilon_0)$. If we plug this back into the equation and exponentiate we obtain
\[
 (6/\varepsilon_0)^d \leq 2^{2\alpha (6/\varepsilon_0)^{1/\alpha}}
\]
and, by monotonicity, this estimate remains valid for all $0<\varepsilon \leq \varepsilon_0$.
Consequently, we find $C$ such that $N_{\varepsilon}(\Rs{\alpha,2},L_{\infty}) \leq 2^{k-1}$ for $\varepsilon = Ck^{-\alpha}$, provided $k \geq c_{\alpha} d\log d$ with the constant $c_{\alpha} = (C/6)^{1/\alpha}$ 
independent of $k$ and $d$. We conclude $e_k(\Rs{\alpha,2},L_{\infty}) \lesssim k^{-\alpha}$.

Observe that for $\varepsilon_0^{1/\alpha} = 6^{1/\alpha}(d\log d)^{-1}$ we have
\[
 \log\del{(6/\varepsilon_0)^{1/\alpha}} = \log d + \log\log d \leq 2 \log d
\]
which implies $\log d \geq \frac {1}{2\alpha} \log(6/\varepsilon_0)$. Therefore, $(6/\varepsilon_0)^{1/\alpha} \geq (2\alpha)^{-1} d \log(6/\varepsilon_0)$. The estimate holds also true for all $\varepsilon \leq \varepsilon_0$. To see this put $h(x) = \alpha x^{1/\alpha} (\log x)^{-1}$. The function $h(6/\varepsilon)$ is decreasing for all $0 < \varepsilon < 3\cdot 2^{1-\alpha}$, as
\[
 \od{}{\varepsilon}\del{h(6/\varepsilon)}  = -6\frac{ (6/\varepsilon)^{1/\alpha -1} \del{\log(6/\varepsilon)-\alpha}}{\del{\varepsilon \log(6/\varepsilon)}^2}
\]
is negative for $\log \frac{6}{\varepsilon} > \alpha$, i.e. $\varepsilon < 3\cdot 2^{1-\alpha}$.
Hence $\alpha (6 /\varepsilon)^{1/\alpha} \del{\log 6 / \varepsilon}^{-1} \geq d/2$ for $\varepsilon < \varepsilon_0$, provided $d > 1$.
This yields
\[
 \log N_{\varepsilon}\del{\Rs{\alpha,2},L_{\infty}} \leq d \log(6/\varepsilon) + \frac{2c}{\varepsilon} \leq \frac{2 \cdot 6^{1/\alpha} + 2c}{\varepsilon^{1/\alpha}}.
\]
Consequently, we find $C > 0$ such that $N_{\varepsilon}\del{\Rs{\alpha,2},L_{\infty}} \leq 2^{k-1}$ holds true for $\varepsilon = C k^{-\alpha}$, provided $k \geq \del{\frac C6}^{1/\alpha} d\log(d)$. We conclude
\[
 e_k(\Rs{\alpha,2},L_{\infty}) \lesssim k^{-\alpha},
\]
provided $k \geq c_{\alpha} \del{d\log d}$. Note that the constant $c_{\alpha} =(C/6)^{1/\alpha}$ is independent of $d$.
\end{detail}
\end{proof}
Summarizing this section, the classes of ridge functions with Lipschitz profiles of order $\alpha$ are essentially 
as compact as the class of univariate Lipschitz functions of order $\alpha$. Consequently, when speaking in terms of metric
entropy, these classes of functions must be much smaller than the class of multivariate Lipschitz functions of order $\alpha$.

\section{Sampling numbers of ridge functions}
\label{sec:sampling}

In light of Section \ref{sec:entropy}, one is led to think that efficient sampling of ridge functions should be feasible. Moreover, their simple, two-component structure naturally suggests a two-step procedure: first, use a portion of the available function samples to identify either the profile or the direction; then, use the remaining samples to unveil the other component.

However, in Subsection \ref{sec:sampling1}, we learn that for ridge functions in the class $\Rs{\alpha,p}$, sampling is almost as hard as sampling of general multivariate Lipschitz functions on the Euclidean unit ball. In particular, such two-step procedures as sketched above cannot work in an efficient manner. It needs additional assumptions on the ridge profiles or directions. We discuss this in Subsection \ref{sec:sampling2}.

\subsection{Sampling of functions in $\Rs{\alpha,p}$}
\label{sec:sampling1}

As usual, throughout the section let $\alpha > 0$ be the Lipschitz smoothness of profiles, $s = \llfloor \alpha \rrfloor$ the order up to which derivatives exist, and let $0 < p \leq 2$ indicate the $p$-norm such that ridge directions are contained in the closed $\ell_p^d$-ball.

The algorithms we use to derive upper bounds are essentially the same as those which are known to be optimal for general multivariate Lipschitz functions. Albeit, the ridge structure allows a slightly improved analysis at least in case $p < 2$. 

\begin{proposition}\label{res:sampling_upper_bound} Let $\alpha>0$ and $0<p\leq 2$. For $n \geq {d+s \choose s}$ sampling points the $n$-th sampling number is bounded from above by
\be\label{eq:sampl}
g_{n,d}^{\text{lin}}(\Rs{\alpha, p},L_{\infty}) \leq e_{k-\Delta}(\bar B_2^d, \ell_{p'}^d)^{\alpha}  \,,
\ee 
where $k = \lfloor \log n \rfloor +2$, $\Delta = 1+\lceil \log {d+s \choose s} \rceil$, and $p'$ is the dual index of $p$.
\end{proposition}

\begin{proof}

\emph{Case $\alpha \leq 1$:} In this case, $s=0$ and $\Delta = 1.$
We choose sampling points $x_1, \dots, x_{2^{k-2}}$ such that they form an $\varepsilon$-covering of $\bar B_2^d$
in $\ell_{p'}^d$. Given this covering, we construct (measurable) sets $U_1, \dots, U_{2^{k-2}}$ such that
$U_i \subseteq x_i+\varepsilon \bar B_{p'}^d$ for $i=1,\dots,2^{k-2}$ and
\[
  \bigcup_{i=1}^{2^{k-2}} \Bigl(x_i+\varepsilon \bar B_{p'}^d\Bigr) = \bigcup_{i=1}^{2^{k-2}} U_i, \qquad U_i \cap U_j = \emptyset \; \text{for } i \neq j.
\]
Now we use piecewise constant interpolation: we approximate $f = g(a\cdot ) \in \Rs{\alpha,p}$ by $S f := \sum_{i=1}^{2^{k-2}} f(x_i) \1_{U_i}$. Then,
\begin{align}\label{f7}
 \norm{f - Sf}_{\infty} &= \sup_{i=1,\dots,2^{k-1}} \sup_{x \in  U_i} \abs{f(x) - f(x_i)}\\
 &\leq \sup_{i=1,\dots,2^{k-2}} \sup_{x \in  U_i} \norm{g}_{\Lip_\alpha} \norm{a}^\alpha_p \norm{x-x_i}_{p'}^{\alpha} \leq \varepsilon^{\alpha}.
\end{align}
The smallest $\varepsilon$ is determined by the $(k-1)$-th entropy number $e_{k-1}(\bar B_2^d, \ell_{p'}^d)$. Consequently,
\be\label{f8}
 g^{\text{lin}}_{n,d}(\Rs{\alpha,p},L_{\infty}) \leq g^{ \text{lin}}_{2^{k-2},d}(\Rs{\alpha,p},L_{\infty})\leq e_{k-1}(\bar B_2^d,\ell_{p'}^d)^{\alpha}.
\ee

\emph{Case $\alpha > 1$:} 
We choose the sampling points $x_1, \dots, x_{2^{k-\Delta-1}}$ and the sets $U_1,\dots,U_{2^{k-\Delta-1}}$ as above.
However, instead of piecewise constant interpolation we apply on each of the sets $U_i \subseteq x_i+\varepsilon \bar B_{p'}^d$ a Taylor formula of order $s$ around the center $x_i$.

That is, to approximate a given $f=g(a\cdot ) \in \Rs{\alpha,p}$ we set $Sf := \sum_{i=1}^{2^{k-\Delta-1}} T_{x_i,s} f \1_{U_i}$. Then, by Lemma \ref{lem:taylorapprox} (ii), we have
\begin{align*}
 \norm{f - Sf}_{\infty} &= \sup_{i=1,\dots,2^{k-\Delta-1}} \sup_{x \in  U_i} \abs{f(x) - T_{x_i,s}f(x)} \leq \frac{1}{s!} \norm{x - x_i}_{p'}^{\alpha} \leq \varepsilon^{\alpha}.
\end{align*}
It takes $2^{k-\Delta-1} {d+s \choose s} \leq n$ function values to approximate all the $T_{x_i,s}$ above up to arbitrary precision by finite-order differences,
cf.\ \cite{V2013}.

The smallest $\varepsilon$ is now determined by the $(k-\Delta)$-th entropy number $e_{k-\Delta}(\bar B_2^d,\ell_{p'}^d)$. We conclude
\begin{align}
 g^{\mathrm{lin}}_{n,d}(\Rs{\alpha,p},L_{\infty}) \leq  g^{\mathrm{lin}}_{2^{k-\Delta-1},d}(\Rs{\alpha,p},L_{\infty}) \leq e_{k-\Delta}(\bar B_2^d, \ell_{p'}^d)^{\alpha}.  
\end{align}

\end{proof}

We turn to an analysis of lower bounds for the classes $\Rs{\alpha,p}$. Our strategy is to find ``bad'' directions which map, for a given budget $n \in \N$,  all possible choices of $n$ sampling points to a small range of $[-1,1]$. There, we let the ``fooling'' profiles be zero; outside of that range, we let the profiles climb as steep as possible. Proposition \ref{res:sampling_lower_bound} below 
states the lower bound that results from this strategy, provided that the ``bad'' directions are given by some  $\mathcal M \subseteq \bar B_p^d \setminus \{0\}$. We discuss appropriate choices of $\mathcal M$ later. In the sequel, we use the mapping $\Psi:\R^d\setminus\{0\}\to \Sphere_2$ defined by $x\mapsto x/\|x\|_2$\,.

\begin{proposition}\label{res:sampling_lower_bound}
Let $\alpha>0$, $0<p\leq 2$, and $\mathcal M \subseteq \bar B_p^d \setminus \{0\}$. Then, for all natural numbers $k$ and $n$ with $n\leq 2^{k-1}$, we have
\[
 g^{\text{ada}}_{n,d}(\Rs{\alpha,p},L_{\infty}) \geq c_{\alpha} \inf_{a \in \mathcal M} \|a\|_2^{\alpha} \cdot e_k(\Psi(\mathcal M),\ell_2^d)^{2\alpha}\,.
\]
The constant $c_{\alpha}$ depends only on $\alpha$.
\end{proposition}
\begin{proof} 
Let us first describe the ``fooling'' profiles in detail. For each $a \in \mathcal M$ and $\varepsilon<1$, we define a function 
\begin{align}\label{eq:bad_function}
 g_{a,\varepsilon}(t)=\vartheta_\alpha \big[(t-\|a\|_2(1-\varepsilon^2/2))_+\big]^{\alpha} 
\end{align}
on the interval $[-1,1]$. 
The factor $\vartheta_\alpha$ assures that $\|g_{a,\varepsilon}\|_{\Lip_\alpha[-1,1]} = 1$. Put 
$f_{a,\varepsilon}(x) = g_{a,\varepsilon}(a\cdot x)$. By construction, we have that $f_{a,\varepsilon} \in \Rs{\alpha,p}$. Moreover, whenever $x \in \bar B_2^d$ and $a\in \mathcal{M}$ is such that 
\begin{align}\label{f31}
\varepsilon^2 < \norm{x-\Psi(a)}^2_2
\end{align}
then $\varepsilon^2 \leq 2-2(x \cdot \Psi(a))$ and hence
\be\label{f3}
  x \cdot a = \|a\|_2 (x \cdot \Psi(a))  < \|a\|_2(1-\varepsilon^2/2)\,.
\ee
Therefore, \eqref{f31} implies $f_{a,\varepsilon}(x) = 0$.

Now, let  $n \leq 2^{k-1}$ and $S \in \AlgAdap_n$ be an adaptive algorithm with a budget of $n$ sampling points. Clearly, the first sampling point $x_1$ must have been fixed by $S$ in advance. Then, let $x_2, \dots, x_n$ be the sampling points which $S$ would choose when applied to the zero function. Furthermore, let $F(x_1,\dots,x_n) \subseteq \Rs{\alpha,p}$ denote the set of functions that make $S$ choose the very points $x_1,\dots, x_n$. Obviously, we have $f_{a,\varepsilon} \in F(x_1,\dots,x_n)$ if \eqref{f31} holds for every $x_i$, $i=1,...,n$. This is true for some $a\in \mathcal{M}$ if we choose $\varepsilon<e_k(\Psi({\mathcal M}),\ell_2^d)$. For the respective function $f_{a,\varepsilon}$, we have 
in particular $N^{\text{ada}}_S(f_{a,\varepsilon}) = 0$ and hence $S[f_{a,\varepsilon}] = S[-f_{a,\varepsilon}]$.
Consequently,
\be\label{f5}
 \max \big\{
  \|f_{a,\varepsilon} - S[f_{a,\varepsilon}]\|_{\infty},
  \|-f_{a,\varepsilon} - S[-f_{a,\varepsilon}]\|_{\infty}
  \big\} 
 \geq \|f_{a,\varepsilon}\|_{\infty} =  g_{a,\varepsilon}(\|a\|_2) = c_{\alpha} \|a\|_2^{\alpha} \varepsilon^{2\alpha}\,,
\ee
where $c_{\alpha} := 2^{-\alpha} \vartheta_{\alpha}$. Since $\varepsilon$ has been chosen arbitrarily but less than $e_{k}(\Psi(\mathcal M),\ell_2^d)$, we are allowed to replace $\varepsilon$ by $e_{k}(\Psi(\mathcal M),\ell_2^d)$ in \eqref{f5} and get
\[
 \sup_{f \in \Rs{\alpha,p}} \|f - S(f)\|_{\infty} \geq c_{\alpha}  \inf_{a \in \mathcal M} \|a\|_2^{\alpha}\cdot e_{k}(\Psi(\mathcal M),\ell_2^d)^{2\alpha}.
\]
Taking the infimum over all algorithms $S \in \AlgAdap_n$ yields
\[
g^{\text{ada}}_{n,d}(\Rs{\alpha, p},L_{\infty}) \geq c_{\alpha} \, \inf_{a \in \mathcal M} \|a\|_2^{\alpha} \, e_{k}(\Psi(\mathcal M), \ell_2^d)^{2\alpha}.
\]
\end{proof}

\begin{theorem}\label{res:ridge_sampling}
Let $\alpha>0$, $s = \llfloor \alpha \rrfloor$, and $0<p\leq 2$. For the classes $\Rs{\alpha,p}$, we have the following bounds: \\
{\em (i)} The $n$-th (linear) sampling number is bounded from above by
\begin{align*}
 g^{\mathrm{lin}}_{n,d}(\Rs{\alpha,p},L_{\infty}) \leq C_{p,\alpha}
 \begin{cases}
  1 &: n \leq 2d{d+s \choose s},\\  
  \left[ \frac{ \log(1+d/\log n_1)}{\log n_1} \right]^{\alpha(1/\max\{1,p\} - 1/2)} &:  2d{d+s \choose s} < n \leq 2^{d+1}{d+s \choose s},\\
  n^{-\alpha/d} \; d^{-\alpha(1/\max\{p,1\}-1/2)} &: n > 2^{d+1}{d+s \choose s}\,,
 \end{cases}
\end{align*}
where $n_1 = n/{[2{d+s \choose s}]}$, and the constant $C_{p,\alpha}$ depends only on $\alpha$ and $p$.\\\\
{\em (ii)} The $n$-th (adaptive) sampling number is bounded from below by
\begin{align*}
 g^{\mathrm{ada}}_{n,d}(\Rs{\alpha,p},L_{\infty}) \geq c_{p,\alpha}
 \begin{cases}
  1 &: n <d,\\
  \left[ \frac{\log_2 \del{1+ d/(2+\log_2 n)}}{2+\log_2 n}\right]^{\alpha(1/p - 1/2)} &: d \leq n <2^{d-1},\\
  n^{-2\alpha/(d-1)} \; d^{-\alpha(1/p - 1/2)}  &: n \geq 2^{d-1}\,.
 \end{cases}
\end{align*}
The constant $c_{p,\alpha}$ depends only on $\alpha$ and $p$.
\end{theorem}

\begin{proof} \emph{(i)} The upper bound is a direct consequence of Proposition \ref{res:sampling_upper_bound} and Lemma \ref{lem:schuett}. Note 
that, for $k$ and $\Delta$ as in Proposition \ref{res:sampling_upper_bound}, it holds true that $k - \Delta - 2 \leq \log n_1 \leq k - \Delta$. Note also that
\[
 {d+s \choose s}^{\alpha/d} \leq (1+s)^{s\alpha/d} d^{s\alpha/d} \leq ((1+s)e)^{s\alpha}
\]
ensures that the constant $C_{p,\alpha}$ can be chosen independently of $d$ and $n$.

\emph{(ii) Case $n < d$}. Let $\mathcal M = \{\pm e_1,\dots,\pm e_d\}$ be the set of positive and negative canonical unit vectors. Clearly, we have 
$\sharp \mathcal{M} = 2d$ and every two distinct vectors in $\mathcal{M}$ have 
mutual $\ell_2^d$-distance equal to or greater than $\sqrt{2}$. Let $k$ be the smallest integer such that 
$n \leq 2^{k-1}$; this implies $2^{k-1}<2d$. Hence, whenever $2^{k-1}$ 
balls of radius $\varepsilon$ cover the set $\mathcal M$, there is at least one $\varepsilon$-ball which contains two elements from $\mathcal{M}$. In consequence, we have
$2\varepsilon\geq\sqrt{2}$ and hence $e_k(\mathcal{M},\ell_2^d) \geq \sqrt{2}/2$.  By Proposition \ref{res:sampling_lower_bound}
and the fact that $\mathcal{M} = \Psi(\mathcal{M})$, we obtain
\[
 g^{\text{ada}}_{n,d}(\Rs{\alpha,p},L_{\infty}) \geq c_{\alpha} e_k(\mathcal{M},\ell_2^d)^{2\alpha} \geq c_{\alpha} 2^{-\alpha}\,.
\]

\noindent{\em Case $d \leq n < 2^{d-1}$}. For $m\leq d$, consider the subset of $m$-sparse vectors of the $p$-sphere,
\[
 \Sparse_{m,p} = \big\{ x \in \Sphere_p:  \;\sharp\;\supp(x) = m \big\}.
\]
Using the combinatorial construction of \cite{GS}, cf.\ also \cite{FPRU10}, we know that there exist at least $(d/(4m))^{m/2}$ vectors in $\Psi(\Sparse_{m,p})=\Sparse_{m,2}$ having mutual $\ell_2^d$-distance greater than $1/\sqrt{2}$. Therefore, we have 
\begin{equation}\label{f32}
   \ell \leq m/2 \log(d/(4m))\quad \implies \quad e_{\ell}(\Psi(\Sparse_{m,p}),\ell_2^d) \geq \sqrt{2}/4. 
\end{equation}
Let $k$ again be the smallest integer such that $n \leq 2^{k-1}$. Hence, $k\leq d$. Choose
\[
 m^* := \big\lfloor \min \{ 4k/\log(d/(4k)), k \} \big\rfloor\leq k\,.
\]
Because of $k>\log d$, we have $\min\{\log d,4\} \leq m^{*}\leq d$. Put $\mathcal M = \Sparse_{m^*,p}$. 
If $k \leq d/64$, then $\log(d/(4k)) \geq 4$ and  $k \leq m^*\log(d/(4k))/2 \leq m^*\log(d/(4m^*))/2$. Hence, by \eqref{f32},
one has $e_k(\Psi(\Sparse_{m^*,p}),\ell_2^d) \geq \sqrt{2}/4.$
Consequently, by Proposition \ref{res:sampling_lower_bound}, it follows that
\begin{align*}
 g^{\text{ada}}_{n,d}(\Rs{\alpha,d},L_{\infty}) &\geq c_{\alpha} (m^*)^{\alpha(1/2-1/p)} e_k(\Psi(\Sparse_{m^*,p}),\ell_2^d)^{2\alpha} \\
 &\geq c_{\alpha}8^{-\alpha}4^{-\alpha(1/p-1/2)}\Big[\frac{\log(d/(4k))}{k}\Big]^{\alpha(1/p-1/2)}\\
 &\geq c_{\alpha} 8^{-\alpha} 8^{-\alpha(1/p-1/2)} \left( \frac{\log(1+d/k)}{k}\right)^{\alpha(1/p-1/2)}\\
 &\geq c_{p,\alpha}\left( \frac{\log(1+d/k)}{k}\right)^{\alpha(1/p-1/2)}\,.
\end{align*}
On the other hand, if $d/64 < k \leq d$, then $m^*  =  k$. By $\mathds S_2^{k-1} \subset \Psi(\Sparse_{m^*,p}) \subset \Sphere_2$ and Lemma \ref{res:entropy_sphere}, 
we have $e_k(\Psi(\Sparse_{m^*,p}), \ell_2^d) \asymp 1$. Proposition \ref{res:sampling_lower_bound}, together with 
$\log(1+d/k) < 8$ for $k>d/64$, implies 
\begin{align*}
 g^{\text{ada}}_{n,d}(\Rs{\alpha,p},L_\infty) \geq c'_{\alpha}  k^{-\alpha(1/p-1/2)} &\geq c'_{\alpha}8^{-\alpha(1/p-1/2)} \left( \frac{\log(1+d/k)}{k} \right)^{\alpha(1/p-1/2)}\\
 &= c'_{p,\alpha}\left( \frac{\log(1+d/k)}{k} \right)^{\alpha(1/p-1/2)}\,.
\end{align*}

\noindent{\em Case $n \geq 2^{d-1}$}. Again, $k$ is chosen such that $2^{k-2}<n\leq 2^{k-1}$, which implies $k \geq d$. In this case, we choose
$\mathcal M = \Sphere_p$. By Lemma \ref{res:entropy_sphere} and Proposition \ref{res:sampling_lower_bound}, we obtain
\begin{align*}
 g^{\text{ada}}_{n,d}(\Rs{\alpha,p},L_{\infty}) &\geq c_{\alpha} \; d^{-\alpha(1/p-1/2)}e_{k}(\Sphere_2,\ell_2^d)^{2\alpha}\\ 
 &\geq c_{\alpha}d^{-\alpha(1/p-1/2)} \; (4n)^{-2\alpha/(d-1)}\\
 &\geq c_{\alpha}4^{-2\alpha}d^{-\alpha(1/p-1/2)}n^{-2\alpha/(d-1)}\,.
\end{align*}
This completes the proof. 
\end{proof}

\begin{remark}\label{prop:p=2}
Consider the situation $p = 2$.  For sampling numbers with $n \leq 2^{d-1}$, we have
\[
 g^{\text{ada}}_{n,d}(\Rs{\alpha,2},L_{\infty}) \asymp 1.
\]
For sampling numbers with $n \geq 2^{d+1} {d+s \choose s}$, we have
\begin{equation}\label{eq:gap}
n^{-2\alpha/(d-1)}\lesssim g^{\text{ada}}_{n,d}(\Rs{\alpha, 2},L_{\infty}) \lesssim n^{-\alpha/d}.
\end{equation}
The upper estimate on sampling numbers is exactly the same as for multivariate Lipschitz functions from $\Lip_{\alpha}(\Omega)$. Although there is a gap between lower and upper bound in \eqref{eq:gap},
the factor $1/(d-1)$ in the exponent of the lower bound allows us to conclude that
sampling of ridge functions in $\Rs{\alpha,2}$ is nearly as hard as sampling of general Lipschitz functions from $\Lip_{\alpha}(\Omega)$. Hence, we have the opposite situation to Section \ref{sec:entropy}, where ridge functions in $\Rs{\alpha,2}$ 
behave similar to univariate Lipschitz functions.
\end{remark}

\begin{remark}\label{rem:CDDKP}
Let us consider the modified ridge function classes $\tilde{\mathcal R}_d^{\alpha,p}$ and $\bar{\mathcal R}_{d}^{\alpha,p}$ defined by 
\be\label{mod_ridge}
 \tilde{\mathcal R}_d^{\alpha,p} := \big\{ f: [0,1]^d \to \R\,:\, f(x) = g(a \cdot x), \; \|g\|_{\Lip_{\alpha}[0,1]} \leq 1 ,\;  \|a\|_p \leq 1, \; a \geq 0 \big\},
\ee
for $0<p\leq 1$, and 
\be\label{mod_ridge2}
 {\bar{\mathcal{R}}}_{d}^{\alpha,p} := \big\{ f: \bar B^d_2 \cap [0,1]^d\,  \to \R\,:\,f(x) = g(a \cdot x), \; \|g\|_{\Lip_{\alpha}[0,1]} \leq 1 ,\;  \|a\|_p \leq 1, \; a \geq 0 \big\}.
\ee
for $0<p\leq 2$. Here, $a\ge 0$ means, that all coordinates of $a$ are non-negative.\\
{\em (i)} In the recent paper \cite{CDDKP} it has been shown that there is an adaptive algorithm which attains a decay rate of $n^{-\alpha}$ for the worst-case $L_{\infty}$-approximation error with respect to the class $\tilde{\mathcal R}_d^{\alpha,1}$, provided that $n \geq d$. In terms of adaptive sampling numbers (such that the feasible algorithms are adjusted to the domain $[0,1]^d$), this reads as
\begin{align}\label{eq:ada_helps}
 g^{\text{ada}}_{n,d}(\tilde{\mathcal R}_d^{\alpha,1}, L_{\infty}) \leq C_{\alpha}n^{-\alpha}, \quad n \geq d. 
\end{align}

At the same time, a careful inspection of the proofs of Propositions \ref{res:sampling_upper_bound}, \ref{res:sampling_lower_bound}, and Theorem \ref{res:ridge_sampling} shows that the results can be carried over to the classes $\bar{\mathcal{ R}}_d^{\alpha,p}$ for all $0<p\leq 2$.  In particular, for $0<p\leq 1$, we have the lower bound
\begin{align}\label{eq:ada_doesnt_help}
 g^{\text{ada}}_{n,d}(\bar{\mathcal R}_d^{\alpha,p}, L_{\infty}) \geq c_{p,\alpha} n^{-2\alpha/(d-1)}d^{\alpha(1/2-1/p)} , \quad n \in \N\,.
\end{align}
The estimates \eqref{eq:ada_helps} and \eqref{eq:ada_doesnt_help} look conflicting at first glance.
We encounter the rather surprising phenomenon, that enlarging the domain of the class of functions under consideration
leads to better approximation rates.
To understand this, let us briefly sketch the adaptive 
algorithm of \cite{CDDKP}. For $f = g(a\cdot) \in \tilde{\mathcal R}_d^{\alpha,p}$ not the zero function, the idea is to first sample along 
the diagonal of the first orthant, that is, at points $x = t (1,\dots,1)$ with $t \in [0,1]$. Importantly, it is guaranteed that we can 
take samples from the whole relevant range $[0,\|a\|_1]$ of the profile $g$ of $f$. This in turn assures that, by sampling adaptively 
along the diagonal, we find a small range in $[0,\|a\|_1]$ where the absolute value of $g'$ is strictly larger than 0. Then, the ridge 
direction $a$ can be recovered in a similar way as we do in Subsection \ref{sec:sampling2}.

On the other hand, for the classes $\bar{\mathcal R}_d^{\alpha,p}$, this adaptive algorithm will not work. Assume we sample again 
along the (rescaled) diagonal. This time, we can be sure that we are able to reach every point in the intervall $[0,\|a\|_1/\sqrt{d}]$. But this interval is in most cases strictly included in the relevant interval $[0,\|a\|_2]$ for $g$. Hence, it is not guaranteed anymore that we sample the whole relevant range of $g$ and find an interval on which $g'$ is not zero.

(ii) Admittedly, the domain $\Omega = [0,1]^d \cap B^d_2$ in \eqref{mod_ridge2} is a somewhat artificial choice in case of $p\leq1$, whereas the cube $\Omega = [0,1]^d$ seems natural. Conversely, the definition in \eqref{mod_ridge} is not reasonable in case $p>1$, since then $a\cdot x$ might exceed
the domain interval for $g$. However, $\Omega = [0,1]^d \cap B^d_2$ is the natural choice for $p=2$ in \eqref{mod_ridge2}. In this situation, we suffer from the curse of dimensionality for adaptive algorithms using standard information, see Remark \ref{prop:p=2} and Theorem \ref{res:tractability},(1) below. This shows that the condition $p\leq1$ is essential in the setting of \cite{CDDKP} and that 
\eqref{eq:ada_helps} can not be true for the class $\bar{\mathcal{R}}^{\alpha,2}_d$.
\end{remark}

\subsection{Recovery of ridge directions}
\label{sec:sampling2}

We return to the question under which conditions the two-step procedure sketched at the beginning of Section \ref{sec:sampling} is successful. The adaptive algorithm of \cite{CDDKP}, which we have already discussed in Remark \ref{rem:CDDKP}, first approximates the profile $g$.
Unfortunately, we could already argue that this algorithm cannot work in our setting.  There is an opposite approach in Fornasier et al. \cite{FSV12}, which first tries to recover the ridge direction and conforms to our setting.
Following the ideas of \cite{BP}, the authors developed an efficient scheme using Taylor's formula to approximate ridge 
functions with $C^s$ profile obeying certain integral condition on the modulus of its derivative. This condition was satisfied for example if $\abs{g'(0)} \ge \kappa > 0$.
In their approach, the smoothness parameter $s$ had to be at least $2$. Using a slightly different analysis, this scheme
turns out to work for  Lipschitz profiles of order $\alpha>1$.

Before we turn to the analysis, let us sketch the Taylor-based scheme in more detail. As transposes of matrices and vectors appear frequently, for reasons of convenience, we write $a \cdot x = a^Tx$ for the remainder of this subsection. Now, Taylor's formula in direction $e_i$ yields
\begin{align*}
 f(h e_i) &= f(0) + h \nabla f(\xi^{(i)}_h e_i)^T e_i\\
 &= g(0) + h g'(\xi^{(i)}_h a_i) a_i\,.
\end{align*}
Hence, we can expose the vector $a$, distorted by a diagonal matrix  with components
\[
  \xi_h = (g'(\xi^{(1)}_h a_1), \dots, g'(\xi^{(d)}_h a_d))
\]
on the diagonal. In total, we have to spend only $d+1$ function evaluations for that. Moreover, each of $\xi_h$'s components can be pushed arbitrarily close to $g'(0)$. This gives an estimate $\hat{a}$ of $a/\norm{a}_2$, along which we can now conduct classical univariate approximation. Effectively, one samples a distorted version of $g$ given by
\[
\tilde g: \intcc{-1,1} \rightarrow \R, \; t \mapsto f(t\hat a) = g(t a^T \hat a).
\]
The approximation $\hat g$ obtained in this way, together with $\hat a$, forms the sampling approximation to $f$,
\[
 \hat f(x) = \hat g(\hat a^Tx).
\]
Observe that $\tilde g(\hat a^T x) = g(a^T \hat a \hat a ^T x)$, so it is crucial that $\hat a \hat a ^T$ spans a subspace which is close to the one-dimensional subspace spanned by $aa^T$, in the sense that
\[
 \big\| a^T ( I_d - \hat a \hat a^T) \big\|_2
\]
has to be small. Importantly, this gives the freedom to approximate $a$ only up to a sign.
Finally, let us note that if the factor $g'(0)$ can become arbitrary small, the information we get through Taylor's scheme about $a$ becomes also arbitrarily bad. Hence, for this approach to work, it is necessary to require $\abs{g'(0)} \ge \kappa$.


\begin{lemma}\label{lem:reconstruct_a}
Let $0 < \beta \leq 1$, $0 < \kappa \leq 1$, and $\varepsilon > 0$. Further let $\delta = \frac{\varepsilon \cdot \kappa}{2+\varepsilon}$ and $h = (\delta/2)^{1/\beta}$. 
For any $g \in \Lip^{\kappa}_{1+\beta}(\intcc{-1,1})$ and $a \in \bar B_2^d$ with $a\not=0$ let $f = g(a \cdot )$. Put
\begin{equation}\label{eq:rec_a}
  \tilde{a}_i = \frac{f(he_i) - f(0)}{h}, \quad i=1, \dots, d
\end{equation}
and $\hat a = \tilde a / \|\tilde a\|_2$. Then
\[
 \big\| \sign(g'(0)) \hat a - a/\|a\|_2 \big\|_2 \leq \varepsilon.
\]
\end{lemma}
\begin{proof}

 By the mean value theorem of calculus there exist $\xi_h^{(i)} \in \intcc{0,h}$ such that
\[
 \tilde a_i = g'(\xi_h^{(i)}a_i)a_i.
\]
By H\"older continuity we get
\[
 |g'(\xi_h^{(i)}a_i)-g'(0)| < 2|g'|_{\beta} |a_i|^{\beta} |h|^{\beta} \leq \delta
\]
for all $i=1,\dots,d$. Let us observe that $\delta<\kappa$ and, therefore, $\tilde a\not=0$ and $\hat a$ is well defined. Put $\xi = (g'(\xi_h^{(i)}a_i))_{i=1}^d$.  Then we can write $\tilde a = \mathrm{diag}(\xi)a$. For the norm of $\tilde a$ we get
\begin{align*}
 \|\tilde a\|_2 &\leq \| \mathrm{diag}(\xi)a - g'(0)a\|_2 + |g'(0)| \|a\|_2\\
 &\leq \max_{i=1,\dots,d} |g'(\xi^{(i)}_h a_i) - g'(0)| \|a\|_2 + |g'(0)| \|a\|_2\\
 &\leq (\delta + |g'(0)|) \|a\|_2.
\end{align*}
Analogously, by the inverse triangle inequality $\|\tilde a\|_2 \geq (|g'(0)|-\delta) \|a\|_2$. In particular,
\[
 \big|\|\tilde a\|_2 / \|a\|_2 - |g'(0)| \big| \leq \delta.
\]

Now, writing $\gamma = \sign(g'(0))$, we observe
\begin{align*}
 \big\|\gamma \hat a - a/\|a\|_2 \big\|_2
 &\leq
 \big\| \gamma \hat a - |g'(0)| a / \|\tilde a\|_2 \big\|_2
 + \big\| |g'(0)| a / \|\tilde a\|_2 - a/\|a\|_2 \big\|_2\\
 &=\|\tilde a\|_2^{-1} \big( \| (\mathrm{diag}(\xi) - g'(0) I_d) \ a\|_2 + \big| |g'(0)| - \|\tilde a\|_2/\|a\|_2\big| \ \|a\|_2 \big)\\ 
 &\leq 2\delta \|a\|_2 / \| \tilde a\|_2  \leq  2\delta/(|g'(0)|-\delta) \leq 2\delta/(\kappa - \delta) = \varepsilon.
\end{align*}
\end{proof}

Having recovered the ridge direction, we manage to unveil the one-dimensional structure from the high-dimensional ambient space. 
In other words, recovery of the ridge direction is a \emph{dimensionality reduction} step. What remains is the problem of sampling the profile, which can be done using standard techniques. In combination, this leads to the following result:


\begin{theorem}\label{res:ridge_sampling_derivative}
Let $\alpha > 1$ and $0 < \kappa \le 1$.
\begin{enumerate}
\item[(i)] Let $n\le d-1$. Then $g_{n,d}(\Rs{\alpha,2,\kappa},L_{\infty})=g_{n,d}^{\mathrm{lin}}(\Rs{\alpha,2,\kappa},L_{\infty})=1$.
\item[(ii)] Let $n\ge d+1$. 
Then
\begin{align*}
  c_{\alpha} \cdot n^{-\alpha} \leq g_{n,d}(\Rs{\alpha,2,\kappa},L_{\infty}) \le g_{n,d}^{\mathrm{lin}}(\Rs{\alpha,2,\kappa},L_{\infty}) \leq C_\alpha (n-d)^{-\alpha}
\end{align*}
with constant $c_{\alpha}$ and $C_\alpha$, which depend on $\alpha$ only.
\end{enumerate}
\end{theorem}

\begin{proof}
\emph{(i)} It is enough to show that $g_{n,d}(\Rs{\alpha,2,\kappa},L_{\infty})\ge 1$ for $n\le d-1$. Let us assume that a given (adaptive) approximation
method samples at $x_1,\dots,x_n$ and let us denote by $L$ their linear span. Then ${\rm dim}\,L\le n<d$ and we may find $a\in{\R}^d$ with $\|a\|_2=1$ orthogonal to all $x_1,\dots,x_n.$
Finally, if we define $g(t)=t$, we obtain
\begin{align*}
1&=\|g(a^T\cdot)\|_\infty\le\frac{1}{2}\cdot\Bigl\{\|g(a^T\cdot)-S_n(g(a^T\cdot))\|_\infty+\|-g(a^T\cdot)-S_n(-g(a^T\cdot))\|_\infty\Bigr\}\\
&\le g_{n,d}(\Rs{\alpha,2,\kappa},L_{\infty}).
\end{align*}

\emph{(ii)} Fix some $0 < \varepsilon < 1$. Let $\hat{a}$ denote the reconstruction of $a$ obtained by Lemma \ref{lem:reconstruct_a}, which uses $d+1$ sampling points of $f$.
We estimate $g$ by sampling the distorted version
\[
\tilde g: \intcc{-1,1} \rightarrow \R, \; t \mapsto f(t\hat a) = g(t a^T \hat a).
\]
Re-using the value $g(0)$ which we have already employed for the recovery of $a$, we spend $k=n-d\ge 1$ sampling points and obtain
a function $\hat g$ with $\|\hat g-\tilde g\|_\infty\le \varepsilon:=C'_\alpha k^{-\alpha}\|\tilde g\|_{\Lip_\alpha}$.

Now put $\hat f(x) = \hat g(\hat a^T x)$ as our approximation to $f$. To control the total approximation error, observe that
\[
 | \hat f(x) - f(x) | \leq | \hat g(\hat a^T x) - \tilde g (\hat a^T x) | + | \tilde g(\hat a^T x) - g(a^T x) |=:E_1+E_2.
\]
For the first error term $E_1$, 
we immediately get
\[
 E_1 \leq \| \hat g-\tilde g \|_\infty \leq \varepsilon=C'_\alpha \| \tilde g\|_{\Lip_\alpha} k^{-\alpha}\le C_\alpha' k^{-\alpha}
\]
as $\| \tilde g \|_{\Lip_\alpha} \leq \|a\|_2 \, \|g\|_{\Lip_\alpha}\le 1$.

For the second error term, note that
\begin{align*}
 E_2 &= | g(a^T \hat a \hat a^T x) - g(a^T x) |
 \leq \| g \|_{\Lip_\alpha} \, \| a^T (I_d - \hat a \hat a^T) \|_2 \, \|x\|_2\\
 &\leq \|g \|_{\Lip_\alpha} \, \|x\|_2 \,  \|a\|_2  \, \big\| a^T/\|a\|_2 \ (I_d - \hat a \hat a^T) \big\|_2.
\end{align*}
We do not know the exact value of the subspace stability term $\| a^T / \|a\|_2 \ (I_d - \hat a \hat a^T) \|_2$.
But because $ \hat a \hat a^T$ is the identity in direction of $\hat a$, we have the estimate
\begin{align*}
 \big\| a^T / \|a\|_2 \ (I_d - \hat a \hat a^T) \big\|_2
 & = \big\| \big(a / \|a\|_2 - \sign(g'(0)) \hat a \big)^T \ (I_d - \hat a \hat a^T) \big\|_2\\
 & \leq \| I_d - \hat a \hat a^T \|_{2 \rightarrow 2} \ \big\| a/ \|a\|_2 - \sign(g'(0)) \hat a \big\|_2 \\
 &\leq \varepsilon.
\end{align*}
For the last inequality, we have used Lemma \ref{lem:reconstruct_a} and the fact that $\| I_d - \hat a \hat a^T\|_{2 \rightarrow 2} \leq 1$.
As a consequence,
\[
 E_2 \leq \|x\|_2 \, \|a\|_2 \, \|g\|_{\Lip_\alpha}  \, \varepsilon \leq \varepsilon.
\]
Putting everything together, we conclude 
\[
 \| \hat f - f \|_{\infty} \leq 2 \varepsilon\le 2C_\alpha'k^{-\alpha}.
\]

Let us turn to the lower bound. Assume we are given a feasible approximation method $S_n$ that samples at points $\{x_1,\dots, x_n\}\subset \Omega$. Let $\psi_{k,b}$ be as in the proof of Proposition
\ref{res:en_univariate_Lip_derivative}. 
There is an interval $I'\subset I=[\pi/4-1/5,\pi/4+1/5]$ of length $|I'|=1/(5n)$ such that $I'$ does not contain any of the first coordinates of $x_1,\dots,x_n$; in other words, it is disjoint with $\{x_1\cdot e_1,\dots,x_n \cdot e_1\}$, where $e_1=(1,0,\dots,0)$ is the first canonical unit vector. Furthermore, let $b$ be the center of $I'$,
put $\psi = \psi_{2n,b}$, and $a=e_1.$ Finally, with $\gamma$ as in \eqref{eq:maxcosin}, we write 
\begin{align*}
f(x)&=\sin(x\cdot e_1),\\
f_+(x)&=\sin(x\cdot e_1)+(1-\gamma)\psi(x \cdot e_1),\\
f_-(x)&=\sin(x \cdot e_1)-(1-\gamma)\psi(x \cdot e_1).
\end{align*}


As $S_n(f)=S_n(f_+)=S_n(f_-)$ and all the three functions are in $\Rs{\alpha,2,\kappa}$, we may use the triangle inequality
\begin{align*}
\|(1-\gamma)\psi\|_\infty&=\|(1-\gamma)\psi(e_1\cdot)\|_\infty\\
&\le \frac{1}{2}\Bigl\{ \|(1-\gamma)\psi(e_1\cdot)+f-S_n(f)\|_\infty+\|(1-\gamma)\psi(e_1\cdot)-[f-S_n(f)]\|_\infty\Bigr\}\\
&=\frac{1}{2}\Bigl\{\|f_+-S_n(f_+)\|_\infty+\|f_--S_n(f_-)\|_\infty\Bigr\},
\end{align*}
to conclude that
\[
 g_{n,d}(\Rs{\alpha,2,\kappa},L_{\infty}) \gtrsim n^{-\alpha},
\]
with a constant depending only on $\alpha.$
\end{proof}


\begin{remark}
Once we have control on the derivative in the origin, cf.\ Section \ref{sec:sampling2}, recovery of the ridge direction and approximation of the ridge profile can be addressed independently.
Formula \eqref{eq:rec_a} is based on the simple observation that
\[
\frac{\partial f}{\partial x_i}(0)=g'(0)a_i=g'(0)\<a,e_i\>
\]
might be well approximated by first order differences. Furthermore, this holds also for every other direction $\varphi\in\Sphere_2$, i.e.,
\[
\frac{\partial f}{\partial \varphi}(0)=g'(0)\<a,\varphi\>
\]
can be approximated by differences
\[
\frac{f(h\varphi)-f(0)}{h}.
\]
Taking the directions $\varphi_1,\dots,\varphi_{m_{\Phi}}$ at random (and appropriately normalized), one can approximate the scalar products $\{\<a,\varphi_i\>\}_{i=1}^{m_{\Phi}}$.
Finally, if one assumes that $a\in \bar B_p^d$ for $0<p\le 1$, one can recover a good approximation to $a$ by the \emph{sparse recovery} methods of the modern area of \emph{compressed sensing}.
This approach has been investigated in \cite{FSV12}.

Although the algorithms of compressed sensing involve random matrices, once a random matrix with good sensing properties (typically with small constants of their Restricted Isometry Property) is fixed, the
algorithms become fully deterministic. This allows to transfer the estimates of \cite{FSV12} into the language of information based complexity.

It follows from the results of \cite{FSV12} that if $0<p\le 1$ and
$$
c\kappa^{-\frac{2p}{2-p}}\log d\le m_{\Phi} \le Cd,
$$
for two universal positive constants $c,C$, then a function $f\in\Rs{2,p,\kappa}$
might be recovered with high probability up to the precision
\[
\Bigl[\frac{m_{\Phi}}{\log(d/m_{\Phi})}\Bigr]^{1/2-1/p}+(n-m_{\Phi})^{-2}
\]
using $n>m_{\Phi}$ sampling points.

If $1/p\le 5/2$ and $c'\kappa^{-\frac{2p}{2-p}}\log d\le n \le C'd,$ this implies that
\[
g^{\text{lin}}_{n,d}(\Rs{2,p,\kappa},L_\infty)\lesssim
\Bigl[\frac{n}{\log(d/n)}\Bigr]^{1/2-1/p}
\]
and the same estimate holds if $1/p>5/2$ and $c'\kappa^{-\frac{2p}{2-p}}\log d\le n
\le c''(\log d)^{\frac{1/p-1/2}{1/p-5/2}}$. Finally, if
$c''(\log d)^{\frac{1/p-1/2}{1/p-5/2}}\le n\le C'd$, we obtain
\[
g^{\text{lin}}_{n,d}(\Rs{2,p,\kappa},L_\infty)\lesssim n^{-2}.
\]


\end{remark}
%
%
%
%
\section{Tractability results}
\label{sec:tractability} 
 
For the classification of ridge function sampling by degrees of difficulty, the field of information-based complexity \cite{NW1} provides a family of notions of so-called \emph{tractability}. Despite of their simple structure, ridge functions lead to a surprisingly rich class of sampling problems in regard of these notions: we run across almost the whole hierarchy of degrees of tractability if we vary the problem parameters $\alpha$ and $p$, or add the constraint on the profiles' first derivative in the origin. 

Let us briefly introduce the standard notions of tractability. We say that a problem is \emph{polynomially tractable} if its information complexity $n(\varepsilon,d)$ is bounded polynomially in $\varepsilon^{-1}$ and $d$, i.e. there exist numbers $c,p,q > 0$ such that 
\[
 n(\varepsilon,d) \leq c \, \varepsilon^{-p} \, d^q \mbox{ for all $0<\varepsilon<1$ and all $d\in{\N}$.}
\]
A problem is called 
\emph{quasi-polynomially tractable} if there exist two constants $C,t>0$ such that 
\be\label{eq:qptrac}
    n(\varepsilon,d) \leq C\exp(t(1+\ln(1/\varepsilon))(1+\ln d))\,.
\ee
It is called \emph{weakly tractable} if
\be\label{eq:wtrac}
    \lim\limits_{1/\varepsilon+d\to\infty} \frac{\log n(\varepsilon,d)}{1/\varepsilon + d} = 0\,,
\ee
i.e., the information complexity $n(\varepsilon,d)$ neither depends exponentially on $1/\varepsilon$ nor on $d$. 

We say that a problem is \emph{intractable}, if \eqref{eq:wtrac} does not hold. If for some fixed $0<\varepsilon<1$ the number $n(\varepsilon,d)$ 
is an exponential function in $d$ then a problem is, of course, intractable. In that case, we say that the problem suffers from
{\em the curse of dimensionality}. To make it precise, we face the curse if there exist positive numbers $c, \varepsilon_0, \gamma$ such that
\be\label{eq:curse}
        n(\varepsilon,d) \geq c(1+\gamma)^d\,,\quad \mbox{for all } 0<\varepsilon\leq \varepsilon_0 \mbox{ and infinitely many }d\in \N\,.
\ee       
In the language of IBC, Theorems \ref{res:ridge_sampling} and \ref{res:ridge_sampling_derivative} now read as follows:

\begin{theorem}\label{res:tractability}
Consider the problem of ridge function sampling as defined in Subsection \ref{sec:ibc}. Assume that ridge profiles have at least Lipschitz smoothness $\alpha > 0$; further, assume that ridge directions are contained in the closed $\ell_p^d$-unit ball for $p \in \intoc{0,2}$. Then sampling of ridge functions in the class $\Rs{\alpha,p}$
\begin{enumerate}[label=(\arabic*)]
 \item suffers from the curse of dimensionality if $p=2$ and $\alpha<\infty$, 
 \item never suffers from the curse of dimensionality if $p < 2$,
 \item is intractable if $p<2$ and $\alpha \le \frac{1}{1/p-1/2}$,
 \item is weakly tractable if $p<2$ and $\alpha > \frac{1}{1/\max\{1,p\}-1/2}$,
 \item is quasi-polynomially tractable if $\alpha = \infty$, 
 \item and with positive first derivatives of the profiles in the origin it is polynomially tractable, no matter what the values of $\alpha$ and $p$ are. 
\end{enumerate}
\end{theorem}

To prove Theorem \ref{res:tractability}, we translate Theorem \ref{res:ridge_sampling} into bounds on the information complexity
\[
 n(\varepsilon,d) = \min \{n \in \N: \; g_{n,d}(\Rs{\alpha,p},L_{\infty}) \leq \varepsilon\}.
\]

\begin{lemma}\label{logn_upper_bounds}
Let $p < 2$ and $\alpha > 0$. Set $\eta = \alpha(1/2 - 1/p') = \alpha(1/\max\{1,p\} - 1/2)$ and define
\begin{align*}
\varepsilon_1^U := C_{p,\alpha} \left[ \frac{\log(1 + d/\log d)}{\log d}\right]^{\eta},
& \qquad \varepsilon_2^U := C_{p,\alpha} \del{\frac 1d}^{\eta}.
\end{align*}
 Then there are positive constants $C_0$ and $C_1$ such that
\begin{align*}
 \log n(\varepsilon,d) \leq C_0 + C_1
 \begin{cases}
  \log d &: \varepsilon_1^U \leq \varepsilon \leq 1,\\
  \log d \cdot (1/\varepsilon)^{1/\eta}&: \varepsilon_2^U \leq \varepsilon < \varepsilon_1^U,\\
  \log(1/\varepsilon) \cdot  (1/\varepsilon)^{1/\eta} &: \varepsilon < \varepsilon_2^U.
 \end{cases}
\end{align*}
The constants depend only on $p$ and $\alpha$.
\end{lemma}

\begin{detail}
\begin{proof}
For auxiliary reasons, let us introduce
\[
 \varepsilon_3^U := 2^{-\alpha(d+1)/d} {d+s \choose s}^{-\alpha/d} \varepsilon_2^U,
\]
as well as the shorthand $\eta = \alpha(1/2 - 1/p') = \alpha(1/\max\{1,p\} - 1/2)$.
\begin{itemize}
\item For $N_1 := 2d{d+s \choose s}$ we have $g_{N_1,d} \leq \varepsilon_1^U$. By monotonicity of information complexity, for all $\varepsilon \geq \varepsilon_1^U$
\[
 \log_2 n(\varepsilon,d) \leq \log_2 n(\varepsilon_1^U,d) \leq \log_2 N_1 \leq 1 + s \log(1+s) + (1+s)\log d.
\] 

\item Consider $\varepsilon_2^U \leq \varepsilon < \varepsilon_1^U$. Write
\begin{align*}
 \varepsilon(n) := C_{p,\alpha} \left[ \frac{ \log_2(1+d/\log_2 n_1) }{\log_2 n_1} \right]^\eta \leq C_{p,\alpha} \left[ \frac{ \log_2(1+d)}{\log_2 n_1} \right]^\eta\\
 \Rightarrow \log n \leq 1 + \log {d+s \choose s} + C_{p,\alpha}^{1/\eta} \log(1+d) \varepsilon(n)^{-1/\eta}. 
\end{align*}
There is a natural $n$ with $N_1 < n,n+1 \leq 2^{d+1} {d+s \choose s} =: N_2$ such that
\[
  \varepsilon(n+1) \leq \varepsilon \leq \varepsilon(n).
\]
By monotonicity of information complexity,
\begin{align*}
 \log_2 n(\varepsilon,d)
 & \leq \log_2 n(\varepsilon(n+1),d) \leq \log_2(n+1) \leq 2 \log_2 n\\
 & \leq 2 \big(1 +s\log(s+1) 
 + s\log(d) 
 + C_{p,\alpha}^{1/\eta} \varepsilon(n+1)^{-1/\eta} \log(1+d) \big)\\
 & \leq 2 + 2s \log(s+1)+ 2(s+2C_{p,\alpha}^{1/\eta})  \varepsilon^{-1/\eta} \log d.
\end{align*}

\item Assume $\varepsilon < \varepsilon_3^U$. Find $n,n+1 > 2^{d+1} {d+s \choose s}$ such that $\varepsilon(n+1) \leq \varepsilon \leq \varepsilon(n)$. Let $\varepsilon(n) := C_{p,\alpha} n^{-\alpha/d} d^{-\eta}$. Since $\varepsilon(n) < \varepsilon^U_3 < \varepsilon_2^U$ we have $d < C_{p,\alpha}^{1/\eta} \varepsilon(n)^{-1/\eta}$, which yields
\begin{align*}
 \log_2 n \leq \alpha^{-1} (\log_2 C_	{p,\alpha} + 1) \log_2(1/\varepsilon(n)) (1/\varepsilon(n))^{1/\eta}.
\end{align*}
By monotonicity of information complexity,
\begin{align*}
 \log_2 n(\varepsilon,d)
 &\leq \log_2 n(\varepsilon(n+1),d) \leq \log_2 (n+1) < 2\log_2 n\\
 &\leq 2\alpha^{-1} (\log_2 C_{p,\alpha} + 1) \log_2(1/\varepsilon(n)) (1/\varepsilon(n))^{1/\eta}\\
 &\leq 2\alpha^{-1} (\log_2 C_{p,\alpha} + 1) \log_2(1/\varepsilon) (1/\varepsilon)^{1/\eta}.
\end{align*}

\item Assume $\varepsilon_3^U \leq \varepsilon < \varepsilon_2^U$. Put $B = 2^{\alpha(d+1)/d} {d+s \choose s}^{\alpha/d}$.  Then
\[
  \varepsilon_3^U/B   \leq \varepsilon/B < \varepsilon_3;
\]
thus
\begin{align*}
 \log n(\varepsilon,d) \leq \log n(\varepsilon/B,d) \leq \alpha^{-1} (\log_2 C_{p,\alpha} + 1) \log_2(B/\varepsilon) (B/\varepsilon)^{1/\eta}.
\end{align*}
The constant $B$ can be bounded from above by a universal constant independent of $d$ and $n$. Hence we get
\[
 \log_2 n(\varepsilon,d) \leq \tilde C_{p,\alpha} \alpha^{-1} (\log_2 C_{p,\alpha} + 1) \log_2(1/\varepsilon) (1/\varepsilon)^{1/\eta}.
\]
\item Set
\begin{align*}
 C_0 &:= 2 + 2s\log(1+s),\\
 C_1 &:= 2\max \big\{ s+1, s+2C_{p,\alpha}^{1/\eta}, \tilde C_{p,\alpha} \alpha^{-1} (\log_2 C_{p,\alpha} + 1) \big\}.
\end{align*}
\end{itemize}
\end{proof}
\end{detail}

\begin{lemma}\label{logn_lower_bounds}
Let $p < 2$ and $\alpha > 0$. Put
\begin{align*}
\varepsilon_1^L := c_{p,\alpha} \left[ \frac{\log(1 + d/\log d)}{\log d}\right]^{\alpha(1/p-1/2)},
& \quad \varepsilon_2^L := c_{p,\alpha} \del{\frac 1d}^{\alpha(1/p-1/2)},
& \varepsilon_3^L := 4^{-\alpha} \varepsilon_2^L.
\end{align*}
Then there are universal constants $c_0$, $c_1$, which depend only on $p$ and $\alpha$, such that
\begin{align*}
 \log n(\varepsilon,d) \geq c_0 + c_1
  (1/\varepsilon)^{\alpha^{-1} (1/p  - 1/2)^{-1}}
\end{align*}
for $\varepsilon_3^L \leq \varepsilon < \varepsilon_1^L$.
\end{lemma}

\begin{detail}
\begin{proof}[Proof of Lemma \ref{logn_lower_bounds}]
Put $\gamma= \alpha (1/p - 1/2)$.
\begin{itemize}
\item Consider $\varepsilon_2^L \leq \varepsilon < \varepsilon_1^L$. Put
\[
 \varepsilon(n) = c_{p,\alpha} \left[\frac{\log_2(1 +d/(\log_2 n +2)}{\log_2 n + 2}\right]^{\gamma}
\]
for $d  < n \leq 2^{d-1}$. With $1/(\log_2 n + 2) \geq 1/(d+1) \geq 1/(2d)$ we have
\[
 \varepsilon(n) \geq  c_{p,\alpha} \left[\frac{\log_2(3/2)}{\log n + 2}\right]^{\gamma}.
\]
Now choose $d < n,n+1 \leq 2^{d-1}$ such that $\varepsilon(n+1) \leq \varepsilon \leq \varepsilon(n)$. By monotonicity of information complexity, we have
\begin{align*}
 \log_2 n(\varepsilon,d)
 &\geq \log_2 n( \varepsilon(n),d) \geq \log_2 n \geq 1/2 \log_2(n+1)\\
 &\geq -1 + 1/2c_{p,\alpha}^{1/\gamma} \frac{\log(3/2)}{\varepsilon(n+1)^{1/\gamma}} \geq -1 + 1/2c_{p,\alpha}^{1/\gamma} \frac{\log(3/2)}{\varepsilon^{1/\gamma}}.
\end{align*}

\item Assume $\varepsilon_3^L \leq \varepsilon < \varepsilon_2^L$. Then $\varepsilon_2^L \leq 4^{\alpha}\varepsilon$.
\[
 \log_2 n(\varepsilon,d) \geq \log_2 n(4^{\alpha}\varepsilon,d)  \geq -1 + 1/2 c_{p,\alpha}^{1/\gamma} 4^{-\alpha/\gamma} \frac{\log_2(3/2)} {\varepsilon^{1/\gamma}}.
\]

\item Put $c_0 := -1$, $c_1 := 1/2 c_{p,\alpha}^{1/\gamma_{p,\alpha}} 4^{-\alpha/\gamma}$.

\end{itemize}
\end{proof}
\end{detail}

\begin{proof}[Proof of Theorem \ref{res:tractability}]
(1). For $n \leq 2^{d-2}$, the lower bound in Theorem \ref{res:ridge_sampling} gives
\[ g_{n,d}(\Rs{\alpha,2},L_{\infty}) \geq c_{p,\alpha} =: \varepsilon_0.\]
Hence, $n(\varepsilon,d) \geq 2^{d-2}$ for all $\varepsilon < \varepsilon_0$ and we have the curse of dimensionality.

(2). Since $\alpha_1 > \alpha_2$ implies $\Rs{\alpha_1,p} \subseteq \Rs{\alpha_2,p}$, we can w.l.o.g. assume $\alpha \leq 1$. We choose an arbitrary $\varepsilon_2^U \leq \varepsilon \leq 1$. By Lemma \ref{logn_upper_bounds},
\[
 n(\varepsilon,d) \leq 2^{C_0}  d^{ C_1 \varepsilon^{\alpha^{-1} (1/\max\{1,p\} - 1/2)^{-1}}}.
\]
By our assumption $\varepsilon \geq \varepsilon_2^U$, this is true for all natural $d > (C_{p,\alpha} /\varepsilon)^{\alpha^{-1} (1/\max\{1,p\} - 1/2)^{-1}}$. Hence, the curse of dimensionality does not occur. 

(3). Put $\gamma = \alpha (1/p - 1/2)$. Assume $d \to \infty$ and $\varepsilon_3^L \leq \varepsilon< \varepsilon_2^L$. The latter implies
\[
 \left(\frac{c_{p,\alpha}}{4^\alpha}\right)^{1/\gamma} (1/\varepsilon)^{1/\gamma} \leq d < c_{p,\alpha}^{1/\gamma} (1/\varepsilon)^{1/\gamma}.
\]
This yields
\[
 \frac{ \log_2 n(\varepsilon,d) }{ d+1/\varepsilon } \geq \frac{c_0}{d+1/\varepsilon} + c_1 \frac{ (1/\varepsilon)^{1/\gamma}}{ c_{p,\alpha}^{1/\gamma} (1/\varepsilon)^{1/\gamma} + 1/\varepsilon}.
\]
Assuming that $\alpha \leq 1/(1/p - 1/2)$, we have $\gamma \leq 1$ and thus $1/\varepsilon \leq (1/\varepsilon)^{1/\gamma}$. We conclude that
\[
 \frac{ \log n(\varepsilon,d) }{ d+1/\varepsilon } \geq \frac{c_1}{c_{p,\alpha}^{1/\gamma} + 1 } > 0.
\]
Consequently, the problem is not weakly tractable; and thus intractable.

(4). Put $x= 1/\varepsilon + d$. By Lemma \ref{logn_upper_bounds} and $1/\varepsilon\leq x$, $d \leq x$, we have
\begin{align*}
 \log n(\varepsilon,d) \leq C_0 + C_1  \log(x) x^{\alpha^{-1} (1/\max\{1,p\} - 1/2)^{-1}}.
\end{align*}
Now, if $\alpha > \frac{1}{1/\max\{1,p\}-1/2}$, then $\lim_{x\rightarrow \infty}\ x^{-1} \log n(\varepsilon,d) = 0$.

(5). By embedding arguments it is enough to consider the class $\Rs{\infty,2}$. We approximate the function $f \in \Rs{\infty,2}$ via the 
Taylor polynomial $T_{s,0}f(x)$ in $x^0 = 0$. Lemma \ref{lem:taylorapprox}, (ii) gives for every $s\in \N$ the bound
$$
    \|f-T_{s,0}f\|_{\infty} \leq \frac{2}{s!}\,.
$$
Let $\varepsilon>0$ be given and let $s\in \N$ be the smallest integer such that $2/s! \leq \varepsilon$. Then
$(s-1)! \leq 2/\varepsilon$ and therefore $[(s-1)/e]^{s-1} \leq (s-1)! \leq 2/\varepsilon$. This gives
\be\label{f61}
    (s-1)\ln((s-1)/e) \leq \ln(2/\varepsilon)\,.
\ee
We know from \cite{V2013} that it requires $\binom{s+d}{s}$ function values to approximate the Taylor polynomial up to
arbitrary (but fixed) precision. Hence, using
\eqref{f61}, we see that there is a constant $t>0$ such that 
$$
    \ln n(\varepsilon,d) \leq s\ln(e(d+1)) \leq t(1+\ln(1/\varepsilon))(1+\ln d),
$$
which is \eqref{eq:qptrac}.

(6). From Theorem \ref{res:ridge_sampling_derivative} we can immediately conclude $\varepsilon^{-1/\alpha} \lesssim n(\varepsilon,d) \lesssim \varepsilon^{-1/\alpha}$, where the constants behind ``$\lesssim$'' behave polynomially in $d$. Consequently, sampling of ridge functions in $\Rs{\alpha,2,\kappa}$ is polynomially tractable.
\end{proof}

By Lemma \ref{emb}, we know that $\Rs{\infty,2}$ is a subclass of the unit ball in $C^{\infty}(\Omega)$. Besides,  we know that approximation using function values is quasi-polynomially tractable in $\Rs{\infty,2}$, see Theorem \ref{res:tractability}. What is the respective tractability level in $C^{\infty}(\Omega)\,$? Or, to put it differently: how much do we gain by imposing a ridge structure in $C^{\infty}(\Omega)\,$? The seminal paper \cite{NW_2009_2} tells us that approximation in $C^{\infty}([0,1]^d)$ suffers from the curse of dimensionality when norming the space in the way as we did in \eqref{f62}. In contrast, we will show that sampling in $C^{\infty}(\Omega)$ is still weakly tractable. This is not too much of a surprise: due to the concentration of measure phenomenon, the Euclidean unit ball's volume is getting ``very small'' in high dimensions $d$; its measure scales like $(2\pi e/d)^{d/2}$. Anyhow, the result suggests that one still benefits from supposing a ridge structure; infinitely 
differentiable ridge functions from $\Rs{\infty,2}$ probably can be approximated easier than general functions from the unit ball of $C^{\infty}(\Omega)$. This is not guaranteed, however, because we do not show that one cannot get anything better than weak tractability for the sampling of functions in the unit ball of $C_{\infty}(\Omega)$.

\begin{theorem}\label{res:infty} \label{wtrinf} The sampling problem for $C^{\infty}(\Omega)$, where the error is measured in $L_{\infty}(\Omega)$, is weakly tractable. 

\end{theorem}

\begin{proof} Applying Lemma \ref{lem:taylorapprox}, (i) together with \eqref{f60} we obtain for any $f\in C^{\infty}(\Omega)$ with 
$\|f\|_{C^{\infty}(\Omega)}\leq 1$ and every $s\in \N$ the relation
\begin{align}\nonumber
    |f(x)-T_{s,0}f(x)| &\leq \frac{2}{(s-1)!}\|x\|_1^s\quad,\quad x\in \Omega\,,\\\nonumber
    &\leq \frac{2d^{s/2}}{(s-1)!} \,.
\end{align}
Let $s\in \N$ be the smallest integer such that $2d^{s/2}/(s-1)! \leq \varepsilon$. This leads to
$$
   \frac{1}{\sqrt{d}}\Big(\frac{s-2}{e\sqrt{d}}\Big)^{s-2} \leq \frac{(s-2)!}{d^{\frac{s-1}{2}}} \leq \frac{2}{\varepsilon}\,
$$
which implies 
\be\label{f65}
    (s-2)\ln\Big(\frac{s-2}{e\sqrt{d}}\Big) \leq \ln(2/\varepsilon)+\frac{1}{2}\ln(d)\,.
\ee
To approximate the Taylor polynomial $T_{s,0}f$ with arbitrary precision (uniformly in $f$) we need $\binom{d+s}{s}$ function values, see
\cite[p.\ 4]{V2013}. Let us distinguish two cases. If $(s-2) \leq e^2\sqrt{d}$ we obtain 
$$
  \ln n(\varepsilon,d) \leq s\ln(e(d+1)) \leq (e^2\sqrt{d}+2)\cdot\ln(e(d+1))
$$
and hence \eqref{eq:wtrac}\,. If $s-2 > e^2\sqrt{d}$ then \eqref{f65} yields $s-2 \leq \ln(2/\varepsilon)+\ln(d)$. Thus,
$$
  \ln n(\varepsilon,d) \leq s\ln(e(d+1)) \leq (\ln(2/\varepsilon)+\ln(d)+2)\cdot \ln(e(d+1))
$$
and again \eqref{eq:wtrac} holds true. This establishes weak tractability. 
\end{proof}

\begin{remark} {(i)} The result in Theorem \ref{wtrinf} is also a consequence of the arguments in  \cite[Sections 5.2,
5.3, and Section 6]{HiNoUlWo13} by putting $L_{j,d} = d^{j/2}$. \\
{(ii)} Recently, Vyb\'iral \cite{V2013} showed that there is quasi-polynomial tractability if one replaces the classical norm $\sup_{ \gamma \in \N_0^d} \|D^{\gamma} f\|_{\infty}$ by $\sup_{k\in \N_0}\sum_{|\gamma| = k}\|D^{\gamma}f\|_{\infty}/\gamma !$ in $C^{\infty}([0,1]^d)$. In contrast to that, Theorem \ref{res:infty} shows weak tractability for the classical norm on the unit ball. 

\end{remark}

\textbf{Acknowledgments}
The authors would like to thank Aicke Hinrichs, Erich Novak, and Mario Ullrich for pointing out relations to the paper \cite{HiNoUlWo13}, as well as Sjoerd Dirksen, Thomas K\"uhn, and Winfried Sickel for useful comments and discussions.
The last author acknowledges the support by the DFG Research Center {\sc Matheon} ``Mathematics for key technologies'' in Berlin.

\end{document}